\documentclass[12pt,leqno,a4paper]{amsart}
\usepackage{amssymb,enumerate}
\newcommand{\CC}{{\mathbb{C}}}
\newcommand{\FF}{{\mathbb{F}}}
\newcommand{\ZZ}{{\mathbb{Z}}}

\newcommand{\fA}{{\mathfrak{A}}}
\newcommand{\fS}{{\mathfrak{S}}}

\newcommand{\bG}{{\mathbf{G}}}
\newcommand{\bT}{{\mathbf{T}}}

\newcommand{\cE}{{\mathcal{E}}}

\newcommand{\ad}{{\operatorname{ad}}}
\newcommand{\avg}{{\operatorname{avg}}}
\newcommand{\End}{{\operatorname{End}}}
\newcommand{\Irr}{{\operatorname{Irr}}}
\newcommand{\SC}{{\operatorname{sc}}}
\newcommand{\Id}{{\operatorname{Id}}}
\newcommand{\St}{{\operatorname{St}}}
\newcommand{\tr}{{\operatorname{tr}}}

\newcommand{\GL}{{\operatorname{GL}}}
\newcommand{\PGL}{{\operatorname{PGL}}}
\newcommand{\PSL}{{\operatorname{L}}}
\newcommand{\SL}{{\operatorname{SL}}}
\newcommand{\PGU}{{\operatorname{PGU}}}
\newcommand{\PSU}{{\operatorname{U}}}
\newcommand{\GU}{{\operatorname{GU}}}
\newcommand{\SU}{{\operatorname{SU}}}
\newcommand{\PCSp}{{\operatorname{PCSp}}}
\newcommand{\Sp}{{\operatorname{Sp}}}
\newcommand{\PSp}{{\operatorname{S}}}
\newcommand{\OO}{{\operatorname{O}}}
\newcommand{\SO}{{\operatorname{SO}}}
\newcommand{\GO}{{\operatorname{GO}}}
\newcommand{\PCO}{{\operatorname{PCO}}}
\newcommand{\Spin}{{\operatorname{Spin}}}

\newcommand{\Chevie}{{\sf Chevie}}

\def\pmod#1{~({\rm mod}~#1)}
\newcommand{\tw}[1]{{}^#1\!}
\newcommand\Lem{\text{Lemma}}

\def\skipa{\vspace{-1.5mm} & \vspace{-1.5mm} & \vspace{-1.5mm}\\}

\let\eps=\epsilon

\let\la=\lambda

\newtheorem{thm}{Theorem}[section]
\newtheorem{lem}[thm]{Lemma}

\newtheorem{cor}[thm]{Corollary}

\newtheorem{prop}[thm]{Proposition}

\theoremstyle{definition}
\newtheorem{exmp}[thm]{Example}

\theoremstyle{remark}
\newtheorem{rem}[thm]{Remark}

\begin{document}

\title{Products of conjugacy classes and fixed point spaces}

\date{May 19, 2011; revised January 8, 2011}

\author{Robert Guralnick}
\address{3620 S. Vermont Ave, Department of Mathematics, University of
Southern California, Los Angeles, CA 90089-2532, USA.}
\makeatletter
\email{guralnic@usc.edu}
\makeatother
\author{Gunter Malle}
\address{FB Mathematik, TU Kaiserslautern,
Postfach 3049, 67653 Kaisers\-lautern, Germany.}
\makeatletter
\email{malle@mathematik.uni-kl.de}
\makeatother

\thanks{The first author was partially supported by  NSF grants
  DMS 0653873 and 1001962.}

\begin{abstract}
We prove several results on products of conjugacy classes in finite simple
groups. The first result is that for any finite non-abelian simple
groups,  there  exists a triple of conjugate elements with product $1$ which
generate the group.
This result and other ideas are used to solve a 1966 conjecture of
Peter Neumann about the existence of elements in an irreducible linear group
with small fixed space. We also show that there always exist two conjugacy
classes in a finite non-abelian simple group whose product contains every
nontrivial element of the group. We use this to show that every element in
a non-abelian finite simple group can be written as a product of two $r$th
powers for any prime power $r$ (in particular, a product of two squares
answering a conjecture of Larsen, Shalev and Tiep).
\end{abstract}

\maketitle


\section{Introduction} \label{sec:intro}

Our first main result is the following:

\begin{thm} \label{thm:main}
 Let $G$ be a finite non-abelian simple group.  There exists a conjugacy
 class $C$ of $G$ such that:
 \begin{enumerate}
  \item there is a triple of elements in $C$ which have product $1$ and
   generate $G$ unless $G=\PSL_2(7)$, the projective two linear
   group over the field of $7$ element, and
  \item there exist $x,y \in C$ such that $xy$ is conjugate to $x^2$ and
  generate $G$, unless $G=\PSL_2(q)$ with $q$ even..
 \end{enumerate}
\end{thm} 

The exceptions in the theorem are real exceptions.  Indeed, if a group
satisfies (2), it can have no nontrivial representations of dimension
at most $2$.  

We show that this has the following consequence.

\begin{cor}   \label{cor:main}
 Let $G$ be a finite non-abelian simple group and $k$ an algebraically closed
 field. Let $V$ be a finite dimensional $kG$-module such that $V$ has no trivial
 submodules or quotients and no two dimensional composition
 factors. There exists $g \in G$ such that every eigenspace of $g$ on $V$ has
 dimension at most $(1/3) \dim V$.
\end{cor}

Note that the hypothesis that $V$ has no two dimensional composition factors
is vacuous if the characteristic of $k$ is not $2$. For our intended subsequent
application it is critical that the $g\in G$ that we choose does not depend
on $V$. We also prove a variant of the previous corollary for direct products
of finite simple groups.   

If $V$ is a $G$-module, let $C_V(g)$ denote the fixed space for $g \in G$.
We use the previous result together with some recent results of the first
author and Mar\'oti \cite{GuMa} as well as an improvement in the solvable
case to answer a conjecture of P. Neumann \cite{Nthesis}:

\begin{thm}  \label{thm:main2}
 Let $G$ be a nontrivial irreducible subgroup of $\GL(V)$ where $V$ is a
 finite dimensional vector space. There exists $g \in G$ with
 $\dim C_V(g) \le (1/3) \dim V$.
\end{thm}

This is Theorem~\ref{linear groups}. See Remark~\ref{rem:history} for a short
history of this problem. The example $G=\SO_3(k)$ on its natural module
shows that $1/3$ is best possible. However, if the
dimension of $V$ is large enough, it seems likely that the bound of~$1/3$ can
be improved. Indeed, we prove (Theorem \ref{thm:main3}) that if $\epsilon>0$,
$G$ is finite simple
and $V$ is an irreducible $\CC G$-module of sufficiently large dimension, then
there exist $g \in G$ with all eigenspaces of dimension at most
$\epsilon \dim V$. On the other hand we give examples for suitable non-simple
groups $G$ of irreducible $\CC G$-modules $V$ of arbitrarily large dimension
such that $\dim C_V(g) \ge (1/9) \dim V$ for all $g \in G$ and even with $V$
primitive and $\dim C_V(g) > (1/50) \dim V$ for all $g \in G$, see
Examples~\ref{exmp:A5} and~\ref{exmp:A4}.
\vskip 1pc

We also extend some results of Malle--Saxl--Weigel \cite{MSW}
and Larsen-Shalev-Tiep \cite{LST}
to prove another result about products of
conjugacy classes.

\begin{thm}   \label{thm:main5}
 Let $G$ be a finite non-abelian simple group. 
  There exist conjugacy classes $C_1,C_2$ in $G$ with
   $G=C_1C_2\cup\{1\}$.  Moreover, aside from $G = \PSL_2(q), q =7$
   or $17$, we can assume that each $C_i$ consists of elements of order prime
   to $6$. 
\end{thm}

This immediately implies that any element in a finite non-abelian simple
group is the product of two $m$th powers for $m$  a power of $6$ (answering
a conjecture of Larsen-Shalev-Tiep \cite{LST} for squares -- the  result on squares
has also been obtained independently by Liebeck, O'Brien, Shalev
and Tiep by other methods).   Combining
our methods with  the main result of Chernousov--Ellers--Gordeev \cite{CEG},
we obtain:

\begin{cor}   \label{cor:squares}
 Let $G$ be a finite non-abelian simple group. Let $m$   be either  a prime power
 or a power of $6$. Then every element of $G$ is a product of two $m$th powers.
\end{cor}

The main result of 
 \cite{LST} is  a similar (but asymptotic) result for arbitrary words.  

There is a related conjecture of Thompson that in fact $G=CC$ for some class
$C$. Thompson's conjecture is known to hold in many cases. See Ellers--Gordeev
\cite{EG}.
We show that if $G$ is a finite simple group of Lie type of rank $1$, then
a very strong version of Thompson's conjecture holds (see Theorem~\ref{thm:2B2}
for a precise statement). This shows that a version of the previous corollary
holds for most words for the rank $1$-groups.   See  Theorem \ref{rank1} for
a very strong result for  the Suzuki- and Ree-groups.

Next, we extend some results of Breuer--Guralnick--Kantor \cite{BGK} about
generation of finite simple groups. The proof depends upon knowledge
of maximal subgroups as well as the ideas from the proof of Theorem 
\ref{thm:main5}.

An easy  consequence (see Corollary \ref{getz}) answers a question of Jayce
Getz \cite{getz} in regard to an application to non-solvable base change for
automorphic representations of $\GL_2$:

\begin{thm}   \label{thm:main4}
 Let $S$ be a finite non-abelian simple group other than $\OO^+_8(2)$. There
 exists a conjugacy class $C$ of $S$ consisting of elements whose order is
 prime to~$6$ such that if $1\ne s \in S$, then $S = \langle g,s\rangle$ for
 some $g\in C$.
\end{thm}

The results in \cite{BGK} show there is a conjugacy class $C$ as above (but
the elements in $C$ may not have order prime to $6$). It is proved there that
the probability that a random element of $C$ generates with $s$ is typically
large.   If  $S = \OO^+_8(2)$, then one can take $C$ to consist of elements
of order $15$ (and so with no exceptions we can take $C$ to be a class of 
elements of odd order).


We also show how our methods can be used to answer a conjecture of
Bauer--Catanese--Grunewald \cite{bcg1,bcg2} regarding Beauville structures
(this is related to the existence of certain free actions of a finite group
acting on a product of two smooth projective curves -- see \cite{shelly} for
background).
The conjecture is that all non-abelian simple groups other than $\fA_5$ admit
an unmixed Beauville structure. It has recently been proved by
Garion--Larsen--Lubotzky \cite{GLL} to hold at least asymptotically.
In particular, we prove Theorem~\ref{thm:beau} that the groups $E_8(q)$ admit
an unmixed Beauville-structure.  In a sequel \cite{GMa}, we   use our methods to
solve the conjecture in the affirmative.
\vskip 1pc

We now describe some of the main ingredients in the proofs of our main
results. There usually are two parts: first, we need to show that triples
of elements from specified conjugacy classes $C_i$ and with product~1 exist in
a given group $G$. There is a well-known character formula which counts the
number of such triples, involving values of the complex irreducible characters
of $G$ on the $C_i$. In our situation, we choose the $C_i$ such that only few
irreducible characters vanish simultaneously on all three classes. Then we
may estimate the structure constant using Deligne--Lusztig theory, or for
some small rank groups compute it from known character tables.
\par
Secondly, we need to argue that some of these triples do generate $G$. So we
will choose conjugacy classes $C_i$ whose elements are contained in few
maximal subgroups of $G$, and for which the character formula allows to
estimate the structure constant to be larger than the possible contributions
from maximal subgroups. 
The most delicate case occurs in Theorem~\ref{thm:main} where there is
only one class to be chosen. 
\vskip 1pc

The paper is organized as follows. In Section~\ref{sec:GPPS}, which may be of
independent interest, we classify the maximal subgroups of the simple groups
of Lie type containing certain elements with large irreducible submodules.
In Section~\ref{sec:Lie type}, we prove Theorem \ref{thm:main} for groups of
Lie type. In Section~\ref{sec:spor}, we
complete the proof of that result for alternating and sporadic groups. 
\par
In Section~\ref{peterconj}, we prove Neumann's conjecture Theorem
\ref{thm:main2}. In Section~\ref{char0}, we show a much stronger
(asymptotic) version of the theorem for finite simple groups in characteristic
zero. We also give examples to show that there are large dimensional
examples (in all characteristics) where all fixed spaces are larges.
In Section~\ref{sec:cover}, we prove Theorem \ref{thm:main5} and
Corollary~\ref{cor:squares}. We also derive stronger results for the low rank
groups of Lie type and a very strong result for the image of word maps for
Suzuki and Ree groups. In the final section, we prove Theorem~\ref{thm:main4}
and some corollaries including the application needed for non-solvable base
change for $\GL_2$. The final result is a proof that $E_8(q)$ admits a
Beauville structure for all $q$.
\par
We will use standard notation for the finite simple groups: for the exceptional
groups (and their twisted analogs) we use the Lie notation;  for the simple
classical groups, we prefer to use the notation $\PSL$, $\PSU$, $\PSp$ and
$\OO^{\epsilon}$ for the (projective special) linear, unitary, symplectic
and orthogonal groups.

\section{Zsigmondy primes}  \label{sec:GPPS}

Here we collect some results on subgroups of groups of Lie type containing
large Zsigmondy prime divisors. 
If $q$ is a prime power and $e >2$ is a positive integer, we let 
$\Phi_e^*(q)$ be the largest divisor of $q^e-1$ that is relatively
prime to $q^m-1$ for all $1\le m<e$.  Note that every prime divisor
of $\Phi_e^*(q)$ is congruent to $1$ modulo $e$.   By Zsigmondy's theorem,
$\Phi_e^*(q)>1$ unless $n=6$ and $q=2$ (and indeed, this is true for $e=2$
as well unless $q$ is a Mersenne prime). In particular, aside from that case,
$\Phi_e^*(q) \ge e +1$. 

Hering \cite{He} showed that usually $\Phi_e^*(q)$ is reasonably large (see
Bamberg--Penttila \cite[Lemma 6.1]{BP} for the version we are stating below
and see Feit \cite{F} as well).

\begin{lem}   \label{lem:smallZsig}
 Let $q$ be a power of the prime $p$ and $e > 2$ an integer.
 \begin{enumerate}
  \item[\rm(a)] If  $\Phi_e^*(q)=1$, then $q=2$ and $e=6$.
  \item[\rm(b)] If $\Phi_e^*(q)=e+1$, then $q=p$  and
   $q^e=2^4, 2^{10}, 2^{12}, 2^{18}, 3^4, 3^6$ or $5^6$.
  \item[\rm(c)] If  $\Phi_e^*(q)=2e+1$, then either $q=2$ and
   $e=3, 8,$ or $20$, or $q=4$ and $e=3$ or $6$.
 \end{enumerate}
\end{lem}   

The main result of \cite{GPPS} was the classification of all subgroups of
$\GL_n(q)$ whose order is divisible by some prime divisor of
$\Phi_e^*(q)$ with $e > n/2$.   The list becomes much shorter if we insist
that this prime divisor is larger than $2e+1$.  The previous lemma indicates
that almost always such a prime divisor exists.

\begin{thm}   \label{overgroups}
 Let $G=\GL(V)=\GL_n(q)$ where $q=p^a$ with $p$ prime.  Assume that $n > 2$.  
 Let $r$ either be a Zsigmondy prime divisor of $q^e -1$ with $e>n/2$ and
 $r>2e+1$ or a product of two (not necessarily distinct) Zsigmondy prime
 divisors of $q^e-1$. Suppose that $H$ is an irreducible subgroup of $G$
 containing an element of order $r$. Then one of the following holds:
 \begin{enumerate}
  \item  $H$ contains $\SL(V)$, $\SU(V)$, $\Omega^{(\pm)}(V)$ or $\Sp(V)$; 
  \item  $H$ preserves an extension field structure on $V$ (of degree $f$
   dividing $\gcd(n,e)$);  
  \item $H$ normalizes $\GL_n(p^b)$ for some $b$ properly dividing $a$; or 
  \item $H$ normalizes the subgroup $H_0$ given in Table~\ref{tab:zsig}.
 \end{enumerate}
\end{thm}

\begin{proof}
Under our assumptions, $q$ has order greater than $n/2$ modulo $r$. The main
result of \cite{GPPS} is that Examples 2.1--2.9 are all possibilities.
We go through them one at a time.

The groups in Example 2.1 are those in cases (1) and~(3). Example 2.2 are
reducible groups.
Examples 2.3 and 2.5 are excluded by the hypothesis that $r > 2e+1$ or $r^2$
divides $q^e-1$. Example 2.4 is case (2).  

The rest of the possibilities are nearly simple groups. All the examples in
\cite{GPPS} other than those in Table~\ref{tab:zsig} are eliminated by the
hypotheses.
\end{proof}

\begin{table}[htbp]
\caption{Subgroups with large Zsigmondy primes}   \label{tab:zsig}
\[\begin{array}{|c|c||c||c|c|}
\hline
 n&  e& H_0         & \text{condition}& \text{classical overgroup}\\ \skipa \hline \hline 
 4&  4& \tw2B_2(q) &  q=2^{2k+1}&   \Sp_4  \\ \hline
 6&  6& G_2(q)     &  p=2       &   \Sp_6  \\ \hline
 7&  6& G_2(q)     &  p \ne 2   &   \OO_7  \\ 
  &   & ^2G_2(q)   &  q=3^{2k+1}&   \OO_7  \\ 
  &   & \PSU_3(q)  &  p = 3     &   \OO_7  \\ \hline
 8&  6& \PSU_3(q)  &  p \ne 3   &   \OO_8  \\    
  &   & \SL_2(q^3) & \text{always}& \Sp_8\ (\OO_8 \text{ if $q$ is even})\\ 
  &   &  \Spin_7(q)& \text{always}&   \OO_8  \\  \hline        
 9&  6&  \SL_3(q^2)& q \ne 2    &   \SL_9  \\  \hline
\end{array}\]
\end{table}

Note that the condition on $r$ is precisely that $r > 2e+1$. If $r=2e+1$,
there is a relatively short list of extra possibilities but for $r=e+1$
there are many more examples. 
 
We make this explicit in the 
following sections.   Moreover, we will work with each of the classical groups
and pick a specific $e$ (depending upon the type of group and the dimension). 

We also use the results of Hiss--Malle \cite{HM} and the Atlas \cite{Atl} to
help verify what subgroups embed in each classical group we consider.

\subsection{The special linear groups $\SL_n(q)$}

Set $q=p^a$ and let $G=\SL_n(q)$. We record the following results
which  are immediate consequences of results above and in \cite{GPPS}.  
 
\begin{lem}   \label{lem:maxSLn}
 Let $G=\SL_n(q)$, $n>2$, and set $e=n$. Let $C$ be a conjugacy class of
 elements of order~$r$ dividing $\Phi_{e}^*(q)$. Assume that $r > 2e+1$
 and $r$ is divisible by some prime divisor of $\Phi_{ae}^*(p)$. Then the only
 possible maximal subgroups containing an element of $C$ are:
 \begin{enumerate}
  \item  the normalizer of $\SU_n(q^{1/2})$ (if $n$ is odd and $q$ is a
   square); 
  \item the normalizer of $\Omega^{-}_n(q)$ (if $n$ is even and $q$ is odd);
  \item  the normalizer of $\Sp_n(q)$ (if $n$ is even); and
  \item  the normalizer of $\GL_{n/f}(q^f) \cap G$ for $f$ a prime
   divisor of $n$.
 \end{enumerate} 
 Moreover, such a class exists unless $(n,q)$ is one of  
 (6,2), (4,2), (10,2), (12,2), (18,2), (4,3), (6,3), (6,5), (3,2), (8,2),
 (20,2), (3,4) or (6,4).
\end{lem}

\begin{lem}   \label{lem:maxSLn-1}
 Let $G=\SL_n(q)$, $n>3$, and set $e=n-1$. Let $C$ be a conjugacy class of
 elements of order~$r$ dividing $\Phi_{e}^*(q)$.  Assume that $r > 2e+1$
 and $r$ is divisible by some prime divisor of $\Phi_{ae}^*(p)$. Then the only
 possible maximal subgroups containing an element of $C$ are:
 \begin{enumerate}
  \item  the normalizer of $\SU_n(q^{1/2})$ (if $n$ is even and $q$ is a
   square); 
  \item the normalizer of $\Omega_n(q)$ (if $nq$ is odd); and
  \item the stabilizer of a $1$-space or of a hyperplane.
 \end{enumerate}
 Moreover, such a class exists unless possibly $(n,q)$ is one of
 (7,2), (5,2), (11,2), (13,2), (19,2), (5,3), (7,3), (7,5), (4,2),
 (9,2), (21,2), (4,4) or (7,4).
\end{lem}

The only $\SL_n(q)$, $n>2$, where  neither of the previous lemmas applies are
$\SL_3(2)$, $\SL_3(4)$ and $\SL_4(2) \cong \fA_8$.

\subsection{The special unitary groups $\SU_n(q)$.}
The following two results are easy consequences of earlier results. Let
$q=p^a$. 

If $n$ is odd, then we take $e=2n$ and so the only special cases are
when $\Phi_e^*(q)= 2n+1$ or $4n+1$, that is, $(n,q)$ is one of
$(5,2),(9,2),(3,3),(3,4)$, or $(3,5)$ by Lemma~\ref{lem:smallZsig}. In these
cases we use \cite{HM} and
\cite{GPPS} to check the possibilities. This gives:

\begin{lem}   \label{lem:maxSUodd}
 Let $G=\SU_n(q)$, $n>2$ and $n$ odd, $(n,q)\ne(3,2)$. Set $e=2n$. 
 Let $C$ be a conjugacy class of elements of order~$r$ dividing
 $\Phi_{e}^*(q)$. Assume that $r$ is divisible by some prime
 divisor of $\Phi_{ae}^*(p)$. If $M$ is a maximal subgroup of $G$ containing
 an element of order $r$,  then one of the following holds:
 \begin{enumerate}
 \item $M$ is the normalizer of $G \cap \GU_{n/f}(q^f)$, $f$ an odd prime
  divisor of $n$;
 \item $(n,q)=(5,2)$, $C$ consists of elements of order $11$ and $M$ is the
  normalizer of $\PSL_2(11)$; 
 \item  $(n,q)=(9,2)$, $C$ consists of elements of order $19$, and $M$ is the
  normalizer of $J_3$; 
 \item $(n,q)=(3,3)$, $C$ consists of elements of order $7$, and $M=\PSL_2(7)$; or
 \item  $(n,q)=(3,5)$, $C$ consists of elements of order $7$, and $M$ is the
  normalizer of $\fA_7$ ($3$ classes fused in the automorphism group).
\end{enumerate}  
 \end{lem}
 
Similarly, if $n$ is even, we take $e=2(n-1)$ and so the only problematic
cases are when $\Phi_e^*(q)= 2n-1$, $4n-3$ or $1$, so $(n,q)$ is one of $(4,2),
(6,2),(10,2),(4,3),(4,4)$ or $(4,5)$.
Since $\Phi_6^*(2)=1$, we have to exclude the case ($4,2)$. 
 
\begin{lem}   \label{lem:maxSUeven}
 Let $G=\SU_n(q)$ with $n>2$, $n$ even, $(n,q) \ne(4,2)$.
 Set $e=2(n-1)$. Let $C$ be a conjugacy class of elements of order~$r$
 dividing $\Phi_{e}^*(q)$. Assume that $r$ is divisible by
 some prime divisor of $\Phi_{ae}^*(p)$. If $M$ is a maximal subgroup of $G$
 containing an element of order $r$, then one of the following occurs:
 \begin{enumerate}
  \item  $M$ is the stabilizer  of a nondegenerate $1$-space; 
  \item $(n,q)=(6,2)$,  $C$ is a class of elements of order $11$ and
   $M=M_{22}$ ($3$ classes);
  \item $(n,q)=(4,3)$, $C$ is a class of elements of order $7$, and
   $M = 4_2.\PSL_3(4)$ ($2$ classes) or $M=\fA_7$ ($4$ classes); or
  \item $(n,q)=(4,5)$, $C$ is a class of elements of order $7$, and $M=\fA_7$
   (2 classes fused in $\GU_4(5)$).
 \end{enumerate}
\end{lem}

Finally, we deal with $\SU_4(2)$.   The maximal subgroups containing elements
of order $5$ are $\fS_6$ and $2^4.\fA_5$ (see \cite{Atl}).

\subsection{The odd-dimensional orthogonal groups $\Omega_{2n+1}(q)$}

\begin{lem}   \label{lem:maxSOodd}
 Let  $G=\Omega_{2n+1}(q), q=p^a$, $n>2$ and $q$ odd. Let $e =2n$.
 Let $C$ be a conjugacy class of elements of order~$r$ dividing
 $\Phi_{e}^*(q)$. Assume that $r > 2e+1$ and $r$ is divisible by some prime
 divisor of $\Phi_{ae}^*(p)$. Let $M$ be a maximal subgroup of $G$
 intersecting $C$. Then one of the following holds:
 \begin{enumerate}
  \item $M$ is the stabilizer of a nondegenerate $1$-space;
  \item $n=3$, $M$ is the normalizer of $G_2(q)$;
  \item $n=p=3$, $M$ is the normalizer of $\PSU_3(q)$; or
  \item $n=3$, $q=3^{2a+1}$ and $M$ is the normalizer of a Ree group
   $\tw2G_2(q)$.
 \end{enumerate}
 Moreover, such a class exists unless $(2n+1,q)$ is either $(7,3)$ or $(7,5)$.
 In those two cases, the additional possibilities are given in
 Table~\ref{tab:Oodd}.
\end{lem}
 
The entries in Table~\ref{tab:Oodd} can be proved using \cite{Atl} resp.
\cite{HM}.

\begin{table}[htbp]
\caption{Further maximal subgroups of $\Omega_{2n+1}(q)$, $n\ge3$}
  \label{tab:Oodd}
\[\begin{array}{|c||c|r|c|}
\hline
 (2n+1,q)& M    & \#\text{ classes}& r\\ \skipa \hline \hline 
 (7,3)& \fS_9,\ 2^6.\fA_7,\ \Sp_6(2)              & 2,1,2   &    7 \\ \hline 
 (7,5)& \fA_8,\ 2^6.\fA_7,\ \Sp_6(2),\ \PSU_3(3).2& 2,1,2,2 &    7 \\ \hline
\end{array}\]
\end{table}

 
\subsection{The symplectic groups $\Sp_{2n}(q)$}
 
 
Applying the previous results gives:
 
\begin{lem}   \label{lem:maxSp}
 Let $G=\Sp_{2n}(q), q=p^a$, $n\ge2$ with $(n,q) \ne (2,2)$, and set $e=2n$.
 Let $C$ be a conjugacy class of elements of order~$r$ dividing
 $\Phi_{e}^*(q)$. Assume that $r > 2e+1$ and $r$ is divisible by some prime
 divisor of $\Phi_{ae}^*(p)$. Let $M$ be a maximal subgroup of $G$
 intersecting $C$. Then one of the following holds:
 \begin{enumerate}
  \item $M$ is the normalizer of $\Sp_{2n/f}(q^f)$ for $f$ a prime divisor of $n$;
  \item $M$ is the normalizer of $\SU_n(q)$ with $qn$ odd;
  \item $q$ is even and $M=\Omega_{2n}^-(q)$; 
  \item $n=3$, $q$ is even, and $M$ is the normalizer of $G_2(q)$; or
  \item $n=2$, $q=2^{2a+1}$, and $M$ is the normalizer of a Suzuki
   group $\tw2B_2(q)$.
 \end{enumerate}
 Moreover, such a class exists unless $(2n,q)$ is as in Table~\ref{tab:sp}.
 In those remaining cases, the additional possibilities (for $r|\Phi_e(q)$)
 are listed in the table.
\end{lem}
 
The entries in Table~\ref{tab:sp} follow by the results of \cite{GPPS, HM}.
We excluded $\Sp_4(2) \cong \fS_6$.


\begin{table}[htbp]
\caption{Further maximal subgroups of $\Sp_{2n}(q)$}   \label{tab:sp}
\[\begin{array}{|c||c|r||c|}
\hline
 (2n,q)& M         &     r& \text{Remarks}  \\ \skipa \hline \hline 
 (4,3)& 2^{1+4}.\fA_5               &    5 &   \PSU_4(2) \\ \hline  
 (6,2)& 2^6.\PSL_3(2),\ \fS_8       & 7 & \Phi_{6}^*(2)=1\\ \hline 
 (6,3)& 2.\PSL_2(13)\ (\text{2 classes})&  7 &        \\ \hline    
 (6,4)& \PSL_2(13)                  &   13 &        \\ \hline
 (6,5)& 2.J_2,\ \PSU_3(3)           &    7 &        \\ \hline
 (8,2)& \PSL_2(17)                  &   17 &        \\ \hline
 (10,2)& \text{none}                &   11 &        \\ \hline   
 (12,2)& \fA_{14},\ \PSL_2(25)      &   13 &        \\ \hline 
 (18,2)& \text{none}                &   19 &        \\ \hline
 (20,2)& \PSL_2(41)                 &   41 &        \\ \hline
\end{array}\]
\end{table}


\subsection{The orthogonal groups of plus type $\Omega_{2n}^+(q)$}
Applying the previous results gives:
 
\begin{lem}   \label{lem:maxSO+}
 Let $G=\Omega_{2n}^+(q), q=p^a$, $n>3$, and set $e=2n-2$. Let $C$ be a
 conjugacy class of elements of order~$r$ dividing $\Phi_{e}^*(q)$. Assume
 that $r>2e+1$ and $r$ is divisible by some prime divisor of $\Phi_{ae}^*(p)$.
 Let $M$ be a maximal subgroup of $G$ intersecting $C$.
 Then one of the following holds:
 \begin{enumerate}
  \item $M$ is the stabilizer of a nondegenerate subspace of dimension~1 or~2;
  \item $nq$ is odd and $M$ is the normalizer of $\Omega_n(q^2)$;
  \item $n$ is even, and $M$ is the normalizer of $\SU_n(q)$; or
  \item $n=4$ and $M$ is the normalizer of $\PSU_3(q)$ (for $p\ne 3$) or
   $\Spin_7(q)$.
 \end{enumerate}
 Moreover, such a class exists unless $(2n,q)$ is as in Table~\ref{tab:o+}
 where additional maximal subgroups arise as given.
\end{lem}
    
\begin{table}[htbp]
\caption{Further maximal subgroups of $\Omega_{2n}^+(q)$, $n\ge4$}
  \label{tab:o+}
\[\begin{array}{|c||c|l|r||c|}
\hline
 (2n,q)& M    & \#\text{ classes}& r& \text{Remarks}  \\ \skipa \hline \hline 
 (8,2)& \fA_9,\ 2^6:\fA_8             & 3,3 &   7 & \Phi_6^*(2)=1\\ \hline 
 (8,3)& 2.\OO_8^+(2)                  & 4   &   7 &     \\ \hline  
 (8,4)&  \text{none}                  &     &  13 &     \\ \hline  
 (8,5)& 2.\tw2B_2(8),\ \fA_{10},\ 2.\fA_{10},\ 2.\OO_8^+(2)& 8,4,8,4&  7 & \\ \hline 
 (10,2)& \text{none}                  &     &  17 &     \\ \hline     
 (12,2)& \text{none}                  &     &  11 &     \\ \hline
 (14,2)& \fA_{16},\ \PSL_2(13),\ G_2(3).2&1,2,2& 13 &     \\ \hline
 (20,2)& \PSL_2(19), J_1              &1,1  &  19 &     \\ \hline    
 (22,2)& \text{none}                  &     &  41 &     \\ \hline
\end{array}\]
\end{table}

For the 8-dimensional groups we have used the results of Kleidman \cite{Kl87}.
For $\Omega_{14}^+(2)$ there are two classes of $\PSL_2(13)$: first of all
there are two different (not even quasi-equivalent) 14-dimensional
representations and so at least two classes (and exactly two in the full
orthogonal group). Note that both representations extend to $\PGL_2(13)$
which does sit in the full $\GO_{14}^+(2)$ but not in the simple group.
(The easiest way to see this is that there is an element of order~4 normalizing
the~13 in $\PGL_2(13)$, it permutes freely the 12 nontrivial eigenvalues of
the element of order~13 and on the 2-dimensional fixed space it acts trivially
since it commutes with the element of order 3 which is semisimple regular on
that 2-space. Thus, it has 5 Jordan blocks, but unipotent elements are in the
simple algebraic group if and only if they have an even number of Jordan
blocks.) So there are two classes. 

For $G_2(3)$, there is one class in the full orthogonal group $\GO_{14}^+(2)$
(since there is only one such representation) and it extends to $G_2(3).2$.
The centralizer of an outer involution $x\in G_2(3).2$ (which is unique up to
conjugacy) is $\PSL_2(8).3$. Since $x$ commutes with an element $y$ of order~9
that has precisely two eigenvalues of order~3, it must be trivial on the
2-dimensional space where $y^3=1$ and so it has at least 2 trivial Jordan
blocks.  If it had $t < 6$ Jordan blocks, then the reductive part of its
centralizer (in $\GL_{14}(2)$) would be a direct product of a torus and
$\GL_t(2)$ and so $\PSL_2(8).3$ would embed in $\GL_5(2)$, but the smallest
faithful representation of $\PSL_2(8).3$ is 6 dimensional.
So $x$ has 6 nontrivial Jordan blocks and thus lies in $\Omega_{14}^+(2)$.

\subsection{The twisted orthogonal groups $\Omega_{2n}^-(q)$}
Applying the previous results gives:
 
\begin{lem}   \label{lem:maxSO-}
 Let $G=\Omega_{2n}^-(q),q=p^a $, $n>3$, and set $e=2n$. Let $C$ be a conjugacy
 class of elements of order~$r$ dividing $\Phi_{e}^*(q)$. Assume that $r>2e+1$
 and $r$ is divisible by some prime divisor of $\Phi_{ae}^*(p)$. 
 Let $M$ be a maximal subgroup of $G$ intersecting $C$. 
 Then one of the following holds:
 \begin{enumerate}
  \item $M$ is the normalizer of $\Omega_{2n/f}^-(q^f)$ for $f$ a prime
   divisor of $n$; or
  \item $n$ is odd and $M$ is the normalizer of $\SU_n(q)$.
 \end{enumerate} 
 Moreover, such a class exists unless $(2n,q)$ is as in Table~\ref{tab:o-}
 where additional maximal subgroups arise as given.
\end{lem}


\begin{table}[htbp]
\caption{Further maximal subgroups of $\Omega_{2n}^-(q)$, $n\ge4$}
  \label{tab:o-}
\[\begin{array}{|c||c|r|}
\hline 
 (2n,q)&   M         &     r\\ \skipa \hline \hline 
 (8,2)& \text{none}                       &  17        \\ \hline   
 (10,2)& \fA_{12}                         &  11        \\ \hline    
 (12,2)& \fA_{13},\ \PSL_2(13),\ \PSL_3(3)&  13        \\ \hline
 (18,2)& \fA_{20}                         &  19        \\ \hline       
 (20,2)& \text{none}                      &  41        \\ \hline
\end{array}\]
\end{table}

\subsection{The exceptional groups}
We end this section by completing a result of Weigel \cite{We92} on maximal
subgroups of exceptional groups of Lie type containing certain maximal tori.

\begin{prop}   \label{prop:excelt}
 Let $G$ be a simple exceptional group of Lie type. Then there exists a
 cyclic subgroup $T\le G$ such that $|T|$, $|N_G(T):T|$ and the conjugacy
 classes of maximal overgroups $M\ge T$ in $G$ are as given in
 Table~\ref{tab:exctorus}.
\end{prop}

\begin{table}[htbp] 
  \caption{Cyclic subgroups and maximal overgroups in exceptional groups}
  \label{tab:exctorus}
\[\begin{array}{|l||cc|l|l|}    
\hline
 G& |T|& |A_G(T)|& M\ge T& \text{further maximals}\\
\skipa \hline \hline
 \tw2B_2(q^2),\ q^2\ge8& \Phi_8'& 4& N_G(T)& -\\ \hline
 ^2G_2(q^2),\ q^2\ge27& \Phi_{12}'& 6& N_G(T)& -\\ \hline
 G_2(q),\ 3|q+\eps& q^2+\eps q+1& 6& \SL_3^\eps(q).2& \PSL_2(13)\ (q=4)\\ \hline
 G_2(q),\ 3|q& q^2+q+1& 6& \SL_3(q).2\ (2\times)& \PSL_2(13)\ (q=3)\\ \hline
   \tw3D_4(q)& q^4-q^2+1& 4& N_G(T)& -\\ \hline
 \tw2F_4(q^2),\ q^2\ge8& \Phi_{24}'& 12& N_G(T)& -\\ \hline
 F_4(q),\ 2{\not|}q& q^4-q^2+1& 12& \tw3D_4(q).3& \\ \hline
 F_4(q),\ 2|q& q^4-q^2+1& 12& \tw3D_4(q).3\ (2\times)& \tw2F_4(2),\PSL_4(3).2_2\ (q=2)\\ \hline
       E_6(q)& \Phi_9/(3,q-1)& 9& \SL_3(q^3).3& -\\ \hline
   \tw2E_6(q)& \Phi_{18}/(3,q+1)& 9& \SU_3(q^3).3& -\\ \hline
       E_7(q)& \Phi_2\Phi_{18}/(2,q-1)& 18& \tw2E_6(q)_\SC.D_{q+1}& -\\ \hline
       E_8(q)& \Phi_{30}(q)& 30& N_G(T)& -\\ \hline
\end{array}\]
\footnotesize{Here, $A_G(T):=N_G(T)/C_G(T)$, $\Phi_8'=q^2+\sqrt{2}q+1$,
$\Phi_{12}'=q^2+\sqrt{3}q+1$, $\Phi_{24}'=q^4+\sqrt{2}q^3+q^2+\sqrt{2}q+1$.}
\end{table}

\begin{proof}
The existence of cyclic tori of $G$ of the given orders follows from the
general theory of tori in finite reductive groups, see for example
\cite[\S3.3]{Ca}. In all cases $T$ is (the image in $G$ of) a self-centralizing
maximal torus (of the corresponding finite reductive group of simply-connected
type), and the order of the automizer $A_G(T)=N_G(T)/C_G(T)$ can be
calculated from the Weyl group, see \cite[Prop.~3.3.6]{Ca}. 
\par
The lattice of overgroups of $T$ up to conjugation can be found in the work
of Weigel \cite[Sect.~4]{We92}, except for
$$G\in\{G_2(q),\ F_4(2),\ F_4(3),\ \tw2E_6(2),\ \tw2E_6(3),
      \ E_7(2),\ E_7(3)\}.$$
The maximal subgroups of $G_2(q)$ were determined by Kleidman
\cite[Thm.~A]{Kl88} and Cooperstein \cite[Thm.~2.3]{Co}: for $q\ge5$ the only
maximal subgroups containing $T$ are one class of $\SU_3(q).2$ if
$q\equiv1\pmod3$, respectively one or two classes of $\SL_3(q).2$ else.
For $G_2(3)$ and $G_2(4)$ there is one further class of maximal subgroups
$\PSL_2(13)$. The maximal subgroups of $F_4(2)$ and of $\tw2E_6(2)$ are
printed in the Atlas \cite{Atl}. According to the arguments given in
\cite[Sect.~4(f)]{We92}, the only further overgroup of $T$ in $F_4(3)$ could
be a group with derived group isomorphic to $\SU_3(9)$. But the latter group
contains elements of order~80, while the centralizer of a 5-element in
$F_4(3)$, of shape $C_{10}\Sp_4(3)$ has no element of order~16 (see \cite{Atl}).
(This fact was pointed out to us by Frank L\"ubeck. Alternatively, one can show
that the 25-dimensional module would have to split as $24\oplus1$ for such a
subgroup, whence $\SU_3(9)$ would have to be contained in the stabilizer
$\tw3D_4(3).3$ of that 1-space.)
It is shown in \cite[Thm.~5.1]{LM} that no further overgroups of $T$ arise
in $\tw2E_6(3)$. 
\par
Now assume that $G=E_7(q)$ with $q=2,3$.  If $q=3$, then we take
the cyclic maximal torus $T$ of order $4 \times 703$. If $q=2$, $T$
is cyclic of order $3 \times 57$ (by \cite[Prop.~3.2.2]{Ca}).

We claim that there is a unique maximal subgroup $M$ of $G$ containing 
$T$ (with $M=\tw2E_6(q)_\SC.D_{q+1})$. 
The proof is very similar in spirit to Weigel's proof and is based on various
results of Liebeck and Seitz and others on the maximal subgroups of the
exceptional Chevalley groups.  See \cite{LS} for a summary of these results
and various references.

Let $H$ be a maximal subgroup of $G$ containing $T$ other than $M$.
By \cite[Thm. 3]{LS}, $H$ is an almost simple group. Let $S$ denote its
socle. It follows by \cite[Table 2]{LS} that $S$ must be a Chevalley group
in the same characteristic as $G$. Moreover, by \cite[Thm. 7]{LS}, it
follows that the (untwisted) rank of $S$ is at most $3$.
Inspection of the groups of rank at most~$3$ defined over fields of size a
multiple of $q$ shows as in \cite{We92} that the only possibility is that
$S=\PSU_3(q^3)$.  If $q=3$, we can eliminate the latter case, since the
automorphism group of $S$ contains no elements of order $|T|$.   

So assume that $q=2$ and $F^*(H)=S = \PSU_3(8)$. We will show that the
normalizer of any subgroup of $E_7(2)$ isomorphic to $S$ and containing an 
element of order $19$ in $T$ is contained in $M$. Let $V$ be the irreducible 
$\FF_2G$-module of dimension $56$. Let $T_0 = \langle x \rangle$ be the
subgroup of $T$ of order $19$. Note that $T_0$ has a $2$-dimensional fixed
space on $V$, and so the fixed space of $T_0$ is the same as that of
$M_0=F^*(M)$. Since $M$ acts transitively on those $3$ points, it follows that 
the stabilizer of the vectors are the three subgroups of the form 
$M_0.2 < M$. Looking at the Brauer character table for
$\FF_2\PSU_3(8)$-modules and knowing that the element of order~19 has
Brauer character $-1$, it follows that if $S$ exists, then 
$V$ has three $\FF_2S$-composition factors; one of dimension $54$
and two trivial modules. Since $V$ is self dual, this implies that $S$ has
fixed points on $V$, whence $S \le M$. If $S$ has a $1$-dimensional fixed
space, then the normalizer of $S$ also fixes this vector, whence is also
contained in $M$. If $S$ has a $2$-dimensional fixed space, then $H$ is
contained in the stabilizer of the $2$-space and so again $H \le M$.  
\end{proof}

Note that we proved a bit more for $E_7(2)$: even the cyclic subgroup of order
57 of $T$ is only contained in one maximal subgroup of $E_7(2)$.  We also
note that Ryba \cite{ryba} proved that $\PSU_3(8).6$ in fact does embed
in $E_7(2)$.  It follows that in fact $\PSU_3(8).6$ embeds in
$\tw2E_6(q)_\SC.D_{q+1}$.

\section{Groups of Lie type} \label{sec:Lie type}

This section is devoted to the proof of Theorem~\ref{thm:main} for finite
simple groups of Lie type. We first give a sketch of the approach to be
followed.
Let $G$ be a finite group, $C$ a conjugacy class of $G$.  If $a \in \ZZ$, 
let $C^a$ be the conjugacy class $\{ x^a | x \in C\}$.  
Then, for $a\in\ZZ$ and fixed $x\in C$ the number of pairs
$$n_a(C):=|\{(y,z)\in C\times C^a\mid xyz=1\}|$$
in $G$ is given by the well-known character formula
$$n_a(C)=\frac{|C|^2}{|G|}
  \sum_{\chi\in\Irr(G)}\frac{\chi(C)^2\chi(C^a)}{\chi(1)},$$
where the sum ranges over the complex irreducible characters of $G$ and
$\chi(C)$ denotes the value of $\chi$ on elements of $C$
(see for example \cite[Thm.~I.5.8]{MM}).

Often, for $G$ almost simple and $C$ a large conjugacy class, the main
contribution to this structure constant is by the terms
coming from the linear characters of $G$. In our situation, the class $C$
will be contained in the derived group $G'$, so all linear characters take
value~1 on elements of $C$ and $C^a$. Set $d:=|G/G'|$, the number of linear
characters of $G$. We'll write
$$\eps_a(C)
  :=\frac{1}{d}\sum_{\chi\in\Irr(G)}\frac{\chi(C)^2\chi(C^a)}{\chi(1)}-1
  =\frac{1}{d}\sum_{\chi(1)\ne1}\frac{\chi(C)^2\chi(C^a)}{\chi(1)},$$
so that
$$n_a(C)=d\frac{|C|^2}{|G|}(1+\eps_a(C))
  =d\frac{|G|}{|C_G(x)|^2}\cdot(1+\eps_a(C))$$
for $x\in C$. So to show that there exists triples we need to prove that
$\eps_a(C)<1$. This can either be done using results from Deligne--Lusztig
theory, or, for many small rank groups, using the known generic character
tables which are available in the \Chevie-package \cite{Chevie}.
\par
To prove generation, assume that $H$ is a maximal subgroup of $G$ containing
$x$, and $C_H:=C\cap H$. Then there are at most $|C_H|$ pairs of
elements $y\in C_H$, $z\in H$, such that $xyz=1$.
So we will choose a conjugacy class $C$ of elements whose order is divisible
by a Zsigmondy prime as considered in the previous section, so which are
contained in few maximal subgroups of $G$, and for which the value of
$n_a(C)$ can be estimated from below by the character formula to be larger
than the possible contributions from maximal subgroups.

\subsection{Some character theory}   \label{subsec:3.1}
We first recall some results on character values in finite groups of Lie type.
Let $\bG$ be a connected reductive algebraic group over the algebraic closure
of a finite field and $F:\bG\rightarrow\bG$ a Frobenius map, with finite group
of fixed points $G=\bG^F$. For $\bT$ an $F$-stable maximal torus of $\bG$,
with $T:=\bT^F$, and $\theta\in\Irr(T)$ we denote by $R_{T,\theta}$ the
corresponding Deligne--Lusztig character.

\begin{lem}   \label{lem:DLvalues}
 Let $x\in G$ be a semisimple element. Then we have:
 \begin{itemize}
  \item[\rm(a)] If $R_{T,\theta}(x)\ne0$ then $x\in T^g$ for some $g\in G$.
  \item[\rm(b)] If $x$ is regular, lying in the (unique) maximal
   torus $T$, then $R_{T,\theta}(x)=\pm\theta_T^G(x)$.
 \end{itemize}
\end{lem}

\begin{proof}
Part~(a) is clear from the character formula \cite[Prop.~7.5.3]{Ca}. For the
second part, note that if $x$ is regular, then it lies in a unique $F$-stable
maximal torus, so $C_\bG(x)^\circ=\bT$. The Steinberg character St then takes
value~$\pm1$ on $x$ by \cite[Thm.~6.4.7]{Ca}, so the claim follows by the
fact that St$\cdot R_{T,\theta}=\theta_T^G$ \cite[Prop.~7.5.4]{Ca}.
\end{proof}

Let's translate this to the setting of Lusztig series of characters. For this
let $\bG^*$ be a group in duality with $\bG$, with corresponding Frobenius
map $F^*:\bG^*\rightarrow\bG^*$ and group of fixed points $G^*:={\bG^*}^{F^*}$
(see \cite[\S4.3]{Ca}). There is a bijective correspondence between
$G$-conjugacy classes of pairs $(T,\theta)$ in $G$ as above and $G^*$-classes
of pairs $(T^*,s)$ where $T^*\le G^*$ is a maximal torus in duality with $T$
and $s\in T^*$ is semisimple (see \cite[Prop.~13.13]{DM}). We will also write
$R_{T^*,s}$ for $R_{T,\theta}$ if $(T^*,s)$ corresponds to $(T,\theta)$ in
this way. Then the Lusztig series $\cE(G,s)$ is by definition the set of
constituents of $R_{T^*,s}$ for $T^*$ running over maximal tori of $G^*$
containing $s$. Lusztig has shown that the $\cE(G,s)$, with $s$ running over
semisimple elements of $G^*$ modulo conjugation, form a partition of $\Irr(G)$.
With these definitions we have:

\begin{lem}   \label{lem:Lvalues}
 Let $x\in G$ be semisimple and $\chi\in\Irr(G)$ with $\chi(x)\ne0$. Then
 there exists a maximal torus $T\ni x$ of $G$ and $s\in T^*\le G^*$ such that
 $\chi\in\cE(G,s)$.
\end{lem}

\begin{proof}
By the result of Lusztig, there is some $s\in G^*$ with $\chi\in\cE(G,s)$.
Since $\chi(x)\ne0$ and the characteristic functions of semisimple classes are
uniform, there exists a maximal torus $T^*\le G^*$ with $s\in T^*$ such that
$R_{T^*,s}(x)\ne0$, so $R_{T,\theta}(x)\ne0$ for some $\theta\in\Irr(T)$.
By Lemma~\ref{lem:DLvalues} this implies that $x$ lies in some conjugate
of $T$.
\end{proof}

Let's assume for what follows that $\bG$ has connected center, so that
centralizers of semisimple elements in $\bG^*$ are connected
(see \cite[Thm.~4.5.9]{Ca}). Let $W(s)$ denote the Weyl group of the
centralizer $C_{\bG^*}(s)$ of $s$ in $\bG^*$. The semisimple
character of $G$ corresponding to $s$ is defined (up to sign) by
$$\chi_{s}:=\frac{1}{|W(s)|}\sum_{w\in W(s)} \eps_w\,R_{T_w^*,s}$$
(with certain signs $\eps_w$, see \cite[Def.~14.40]{DM}), where $T_w^*$
denotes a maximal torus of $C_{G^*}(s)$ (hence of $G^*$) parametrized by $w$
in the sense of \cite[Prop.~3.3.3]{Ca}. By \cite[Thm.~8.4.8]{Ca} the degree
of $\chi_{s}$ is given by
$$\chi_{s}(1)=|G^*:C_{G^*}(s)|_{p'},$$
the part of the index of $C_{G^*}(s)$ prime to the characteristic $p$ of $\bG$.
The character $\chi_{s}$ is irreducible (see \cite[Prop.~14.43]{DM}). We let
$W(\theta):=N_G(\bT,\theta)/T$ denote the stabilizer
of $\theta$ in $W(T):=N_G(\bT)/T$. 

\begin{prop}   \label{prop:values}
 Let $x\in G$ be regular, lying in the (unique) maximal torus $T$ of $G$
 parametrized by $v\in W$. Let $(T^*,s)$ correspond to $(T,\theta)$, and assume
 that the intersection of the $F$-conjugacy class of $v$ in $W$ with $W(s)$
 is a single $F$-conjugacy in $W(s)$. Then
 $$\chi_{s}(x)=\pm\sum_{w\in W(T)/W(\theta)}\theta(x^w),$$
 where the sum runs over a set of coset representatives of $W(\theta)$ in
 $W(T)$. In particular $|\chi_{s}(x)|\le |C_W(Fv):C_{W(s)}(Fv)|$.
\end{prop}

\begin{proof}
By assumption $T$ is parametrized by the $F$-conjugacy class of $v\in W$.
Since $x\in T$ is regular, by the definition of $\chi_s$ and
Lemma~\ref{lem:DLvalues} we have
$$\chi_{s}(x)=\frac{1}{|W(s)|}\sum_{w\sim v}\,\eps_v\theta_T^G(x)$$
where the sum extends over those $w\in W(s)$ for which $T_w$ is
$G$-conjugate to $T$, that is, the $F$-conjugates of $v$ in $W(s)$. By
assumption there are exactly $|W(s):C_{W(s)}(Fv)|$ such conjugates, so
$$\chi_{s}(x)=\pm\frac{1}{|C_{W(s)}(Fv)|}\,\theta_T^G(x).$$
 \par
In order to evaluate $\theta_T^G(x)$ observe that for $g\in G$ we have
$x^g\in T$ if and only if $x\in{}^gT$, that is, if $g\in N_G(T)$. Moreover
$N_G(T)=N_G(\bT)$ since $T$ contains the regular element $x$. Thus
$$\theta_T^G(x)=\frac{1}{|T|}\sum_{g\in N_G(\bT)}\theta(x^g)
  =\frac{1}{|T|}\sum_{g\in N_G(\bT)}{}^g\theta(x).$$
Then, with $W(T)=N_G(\bT)/T$ and $W(\theta)=N_G(\bT,\theta)/T$ we have 
$$\theta_T^G(x)=|W(\theta)|\sum_{w\in W(T)/W(\theta)}{}^w\theta(x)
  =|W(\theta)|\sum_{w\in W(T)/W(\theta)}{}\theta(x^w).$$
Since $(T,\theta)$ corresponds to $(T^*,s)$ we have $W(\theta)\cong
N_{G^*}(\bT^*,s)/T^*\cong C_{W(s)}(Fv)$ (see \cite[proof of Prop.~14.43]{DM}),
so
$$\chi_{s}(x)=\pm\frac{1}{|C_{W(s)}(Fv)|}\,|W(\theta)|
  \sum_{w\in W(T)/W(\theta)}{}\theta(x^w)
  =\pm\sum_{w\in W(T)/W(\theta)}{}\theta(x^w).$$
\end{proof}

Apart from these results from Deligne--Lusztig theory we also use the following
result on blocks with cyclic defect groups: let $B$ be a block with
cyclic defect group $P$ generated by $x$, say. Then the character values
$\chi(x)$, for all non-exceptional characters $\chi$ in $B$, are the same up
to sign.

\subsection{The exceptional groups}
We first consider the  simple exceptional group of Lie type. We postpone the
treatment of the classical groups $\tw2G_2(3)'=\PSL_2(8)$ and
$G_2(2)'=\PSU_3(3)$ until the next subsection, as well as the Tits group
$\tw2F_4(2)'$ which will be considered in Section~\ref{sec:spor}.

We show that not all triples of elements from a class of generators of the
cyclic subgroup $T$ as in Table~\ref{tab:exctorus} can lie in proper subgroups.
For the five families of exceptional groups of small rank, the character
tables are known, so that lower bounds for the structure constant $n_a(C)$
can easily be obtained. Using Proposition~\ref{prop:excelt} we can also
estimate the contribution from maximal subgroups containing $T$ from above.
From this one gets:

\begin{prop}   \label{prop:excsmall}
 Theorem~\ref{thm:main} holds for the simple exceptional groups of Lie type
 $\tw2B_2(q^2)$, $\tw2G_2(q^2)$, $G_2(q)$, $\tw3D_4(q)$ and $\tw2F_4(q^2)$,
 with $C$ containing generators of a cyclic subgroup $T\le G$ as in
 Table~\ref{tab:exctorus}.
\end{prop}

\begin{proof}
The character tables of all five families of groups are known explicitly and
contained for example in the {\sf GAP}-package \Chevie\ \cite{Chevie}. With
this system it is possible to compute structure constants generically, that
is, for all prime powers $q$ (in a fixed congruence class modulo~12) at the
same time. For $G=\tw2B_2(q^2)$ one obtains with \Chevie
$$n_a(C)=\frac{|G|}{|T|^2}\bigg(1+\frac{4q^5+11\sqrt{2}q^4+6q^3-2q+\sqrt{2}}
      {\sqrt{2}(q^2+\sqrt{2}q+1)}\bigg)\ge \frac{|G|}{|T|^2}.$$
By Proposition~\ref{prop:excelt} the only maximal subgroup $M$ of $G$
containing $T$ is the normalizer $N_G(T)$, which contributes at most
$|C_H|=|C\cap H|=4$ triples.  \par
Similarly, for $G=\tw2G_2(q^2)$ one finds that
$$n_a(C)\ge \frac{|G|}{|T|^2}.$$
The only maximal subgroup containing $T$, the normalizer $N_G(T)$, contains
at most $6$ of these triples.  \par
For $G=G_2(q)$ one finds that for either congruence
$$n_a(C)\ge \frac{1}{3}q^{10}.$$
For $q\ge5$ the possible number of triples from the subgroups $\SL_3^\eps(q).2$
is too small. For $G_2(3)$ we have $n_1(C)=25456$, $n_{-2}(C)=26185$,
while the two classes of maximal subgroups $\SL_3(3).2$ contain at most 133
triples each and the maximal subgroup $\PSL_2(13)$ contains no more than 15
triples with fixed first component.
Similarly, for $G_2(4)$ we have $n_1(C)=1495561$, $n_{-2}(C)=1499657$,
while the maximal subgroups of type $\SU_3(4).2$ and $\PSL_2(13)$ contain at
most $1380+15$ triples.
\par
From the known character table of $G=\tw3D_4(q)$ one computes that
$$n_a(C)=|T|(q^{12}-4q^6+1)(q^2+1)^2.$$
The only maximal subgroup containing $T$ is the normalizer $N_G(T)$,
contributing 4 triples.
\par
From the known character table of $G=\tw2F_4(q^2)$ one finds that
$n_a(C)\ge q^{42}$ when $q^2\ge8$. Again, the only maximal subgroup containing
$T$ is the normalizer $N_G(T)$, accounting for 12 triples.
\end{proof}

For the larger exceptional groups, neither the character table nor the list of
maximal subgroups is completely known. In order to estimate the structure
constant $n_a(C)$ we use results from Deligne--Lusztig theory as
explained above.

\begin{prop}   \label{prop:excbig}
 Theorem~\ref{thm:main} holds for the simple exceptional groups of Lie type
 $F_4(q)$, $E_6(q)$, $\tw2E_6(q)$, $E_7(q)$, $E_8(q)$ with $C$ containing
 generators of a cyclic subgroup $T$ as in Table~\ref{tab:exctorus}.
\end{prop}

\begin{proof}
Let first $G=F_4(q)$. By Lemma~\ref{lem:Lvalues} the only characters not
vanishing on an element $x\in C$ are those in Lusztig series $\cE(G,s)$ where
$s$ is a semisimple element lying in the dual torus $T^*\le G^*$ (up to
conjugation). As $T$ is the centralizer of any of its
non-identity elements, the same is true for $T^*$. Thus all non-identity
elements $s\in T^*$ are regular, and $\cE(G,s)$ consists of a single
irreducible and thus semisimple Deligne-Lusztig character $\chi_s$
(up to sign). \par
The elements of $\cE(G,1)$ are by definition the unipotent characters of $G$.
Using the character formula \cite[Prop.~7.5.4]{Ca} or the theory of blocks
with cyclic defect group for a prime dividing $|T|$, one finds that
$\chi(x)\in\{0,\pm1\}$ for all unipotent characters $\chi\in\cE(G,1)$, with
$\chi(x)\equiv\chi(1)\pmod{|T|}$. Explicit computation using the degrees of
unipotent characters, given in \cite[13.9]{Ca} for example, shows that the
contribution of these $\chi\ne1_G$ to $\eps_a(C)$ is positive. \par
Next, the character formula in Proposition~\ref{prop:values} shows that for
$s\in T^*$ a regular element,
$\chi_s(x)$ is a sum of $|N_G(T):T|=12$ roots of unity, hence of absolute value
at most~12. There are $(|T|-1)/12$ such semisimple characters, each of degree
$|G:T|_{q'}$, so they contribute at most $12^2|T|^2/|G|_{q'}$ to
$|\eps_a(C)|$. Thus $n_a(C)\ge |G|/|T|^2(1-12^2|T|^2/|G|_{q'})$. \par
By Table~\ref{tab:exctorus} the only maximal subgroups $H$ containing $x$ are
one or two classes of triality groups $\tw3D_4(q).3$, when $q\ge3$. Since
any such subgroup contains the full normalizer of $T$, $T$ is contained in a
unique group in each class, and moreover the intersection $C_H=C\cap H$ is a
single conjugacy class of $H$, whence $|C_H|=|H:T|$. Thus each of the two
classes of triality subgroups contributes at most $|H:T|$ triples, which is
smaller than $n_a(C)$.
For $F_4(2)$, the maximal subgroups above $N_G(T)$ contain less than~$10^8$
triples, while $n_a(C)>10^{13}$.  \par
For $E_6(q)$ we compute the structure constant $n_a(C)$ in the group
$G=E_6(q)_\ad$ of adjoint type, which contains the simple group as a normal
subgroup of index~$d=\gcd(3,q-1)$. Here, the dual group is of type $E_6(q)_\SC$,
and the semisimple elements in the dual torus $T^*$ are either central, or
regular. The regular elements parametrize irreducible Deligne-Lusztig
characters, whose contribution to $|\eps_a(C)|$ is thus bounded above by
$81\,|T|^2/|G|_{q'}$ by Proposition~\ref{prop:values}, while the $d$ central
elements parametrize the $d$ Lusztig series consisting of the various
characters of $G$ having the same restrictions to $E_6(q)$ as unipotent
characters. Explicit computation gives that their contribution to $|\eps(C)|$
is less than $q^{56}/|G:T|$.
\par
By Proposition~\ref{prop:excelt} the only maximal subgroup $H$ containing $T$
is of type $\SL_3(q^3).3$. This subgroup contains at most $|H|<3q^{24}$
triples. \par
The argument for $G=\tw2E_6(q)_\ad$ is completely analogous to the one in the
untwisted case. Here, the contribution from the unipotent characters to
$|\eps_a(C)|$ is less than $2q^{56}/|G:T|$. The only maximal subgroup containing
$T$ is of type $\SU_3(q^3).3$, by Proposition~\ref{prop:excelt}, and we may
conclude as before. \par
For $G=E_7(q)_\ad$ the torus $T^*$ does contain non-regular non-central
elements. We use the results of L\"ubeck \cite[\S5.9]{Lue} on smallest
character degrees. Let $x\in G'=E_7(q)$ be an element of order
$\Phi_2(q)\Phi_{18}(q)/d$, where $d=\gcd(2,q-1)$. The non-trivial character
of $G$ of smallest degree is the unipotent character
$\phi_{7,1}$. It lies in the principal $p$-block for any Zsigmondy prime
divisor $p$ of $\Phi_{18}(q)$, so takes value~$\pm1$ on $x$ (in fact,
value~$-1$). There are just~$d$ characters of this degree, and the next
smallest character degree is
larger than $q^{26}$. Now $|C_G(x)|=(q+1)(q^6-q^3+1)<q^8$, so
$|\chi(x)|<q^4$ for all $\chi\in\Irr(G)$, and moreover clearly there are at
most $q^8$ characters with $\chi(x)\ne0$. So all other characters contribute at
most $q^{20}/q^{26}<q^{-3}$ to $\eps_a(C)$. The only maximal subgroup
containing our element $x$ is $M=\tw2E_6(q).D_{q+1}$, contributing at most
$|M|<q^{80}$ triples.  \par
For $G=E_8(q)$, arguing as for $F_4(q)$ above we see that the non-unipotent
characters
contribute at most $900\,|T|^2/|G|_{q'}$ to $|\eps(C)|$, while the unipotent
characters different from $1_G$ contribute at most $q^{211}/|G:T|$. \par
On the other hand the normalizer $N_G(T)$ contains no more than $30$
triples.
\end{proof}

\subsection{The classical groups}
We now turn to the simple classical groups of Lie type. Let $G_\ad$ be a group
of adjoint type such that $G'=[G_\ad,G_\ad]$ is our given simple group. This
is the group of fixed points under a Frobenius endomorphism of a simple
algebraic group with connected center, so in particular the results of
Section~\ref{subsec:3.1} are applicable. We'll choose the class $C\subset G'$
of an element $x$ whose order is divisible by a Zsigmondy prime as in
Section~\ref{overgroups}, such that its centralizer is a maximal torus of
$G_\ad$. Note that the number of conjugates of a maximal subgroup containing
such an element can grow with both the field size and the rank.  However, note
that the automizer of the maximal torus $T$ containing $x$ acts semiregularly
on the nontrivial powers of $x$. It follows that every nontrivial eigenspace
of $x$ on an irreducible $V$ has dimension
at most $\dim V/[N_G(T):T]$. 


\begin{table}[htbp] 
  \caption{Maximal tori in classical groups} \label{tab:classtorus}
\[\begin{array}{|l|cc|l|}    
\hline
 G& |T|& \!\!|A_G(T)|& \text{ref.}\\
\skipa \hline \hline
 \PGL_n(q),\ n\ge3\text{ odd}& (q^n-1)/(q-1)& n& \Lem~\ref{lem:maxSLn}\\
 \PGL_n(q),\ n\ge4\text{ even}& q^{n-1}-1& n-1& \Lem~\ref{lem:maxSLn-1}\\
 \PGU_n(q),\ n\ge3\text{ odd}& (q^n+1)/(q+1)& n& \Lem~\ref{lem:maxSUodd}\\
 \PGU_n(q),\ n\ge4\text{ even}& q^{n-1}+1& n-1& \Lem~\ref{lem:maxSUeven}\\
 \PCO_{2n+1}(q)& q^n+1& 2n& \Lem~\ref{lem:maxSOodd}\\
 \PCSp_{2n}(q)& q^n+1& 2n& \Lem~\ref{lem:maxSp}\\
 {\PCO_{2n}^\circ}^+(q)& (q^{n-1}+1)(q+1)& 2n-2& \Lem~\ref{lem:maxSO+}\\
 {\PCO_{2n}^\circ}^-(q)& q^n+1& n& \Lem~\ref{lem:maxSO-}\\ \hline
\end{array}\]
\end{table}

We consider the various types case by case.

\begin{prop}   \label{prop:2Dn}
 Theorem~\ref{thm:main} holds for the groups $\OO_{2n}^-(q)$, $n\ge4$, with
 $C$ containing elements of order~$\Phi_{2n}^*(q)$.
\end{prop}

\begin{proof}
We estimate the structure constants $n_a(C)$ in the
group $G:={\PCO_{2n}^\circ}^-(q)$. Since the class $C$ is contained in the
derived subgroup $G'=\OO_{2n}^-(q)$, this will prove the claim. Note that
elements of $C$ are regular since their order is divisble by a Zsigmondy
prime for $q^{2n}-1$. Thus, by Lemma~\ref{lem:Lvalues} the irreducible
characters of $G$ not vanishing
on an element $x\in C$ lie in Lusztig series $\cE(G,s)$ for semisimple
elements $s$ in the dual group $G^*=\Spin_{2n}^-(q)$ whose centralizer order is
divisible by $\Phi_{2n}(q)$. These are precisely the elements in tori
$T^*\le G^*$ of order $q^n+1$; the latter correspond to the $F$-class of an
elements $v$ of the Weyl group $W$ of $G$ with cyclic centralizer of order~$n$.
\par
The elements of $\cE(G,s)$, with $s$ central in $G^*$, are the extensions to
$G$ of the unipotent characters of $G'$. Those which do not vanish on $x$
take value~$\pm1$ on $x$. Furthermore, since $|N_G(T):T|=n$, there are
exactly $n-1$ non-principal unipotent characters of $G'$ not vanishing on $C$;
as all unipotent characters extend to $G$, there are $d(n-1)$ non-linear
characters of $G$ to consider for $\eps_a(C)$, with $d=|G:G'|=\gcd(4,q^n+1)$.
Since the minimal non-trivial character degree of $G$ equals
$b:=(q^n+1)(q^{n-1}-q)/(q^2-1)$ by Tiep--Zalesskii \cite[Thm.~1.1]{TZ}, the
contribution of non-linear characters in $\cE(G,s)$, $s\in Z(G^*)$, to
$|\eps_a(C)|$ is at most
$$(n-1)\frac{1}{b} =\frac{(n-1)(q^2-1)}{(q^n+1)(q^{n-1}-q)}.$$
\par
The non-central elements in $T^*$ have centralizer $Z_r:=\GU_{n/r}(q^r)$ for
some divisor $r$ of $n$ such that $n/r$ is odd. Let $s_r$ denote an element
with centralizer $Z_r$. Then $C_{W}(v)\cong W(T) \cong C_n$,
$W(s)=\fS_{n/r}$, and $C_{W(s)}(v)\cong N_{Z_r}(T)/T \cong C_{n/r}$, so by
Proposition~\ref{prop:values} we have $|\chi_{s_r}(x)|\le r$. By
\cite[\S9]{FoSr2} the $n/r$ characters in $\cE(G,s_r)$ not vanishing on $x$
lie in the same $p$-block as $\chi_{s_r}$, for all Zsigmondy prime divisors~$p$
of $\Phi_{2n}(q)$. In particular, their value on $x$ has the same absolute
value as $\chi_{s_r}(x)$, and similarly for the value on $x^a$. Finally, up
to conjugation there are at most $q^r/r$ elements in $T^*$ with centralizer
$\GU_{n/r}(q^r)$, so the characters in Lusztig series for non-central
elements of $G^*$ contribute at most
$$\sum_{r|n}\frac{n q^r r^3}{r^2|G:Z_r|_{q'}}
  =\sum_{r|n}\frac{nrq^r}{|G:Z_r|_{q'}}$$
to $\eps_a(C)$. Using these estimates an easy calculation now shows that
$\eps_a(C)<0.5$ for all $n\ge4$, $q\ge2$, so that $n_a(C)\ge |C|^2/2|G|$.
\par
The maximal subgroups of $G$ containing an element $x\in C$ were listed in
Lemma~\ref{lem:maxSO-}. For the generic two classes, viz. $H_1=N_G(\SU_n(q))$
(for $n$ odd) and $H_r=N_G(\GO_{2n/r}^-(q^r))$ (for $r$ an odd prime divisor
of $n$), we have $|G:H_1|\ge q^5(q^n+1)$, and $|G:H_r|\ge q^{18}(q^n+1)$,
whence $\sum_{r=1}^n (q^n+1)/|G:H_r|<1/2$. So at most
$$\sum_{r=1}^n |H_r:T|=\frac{|G|}{|T|^2}\sum_{r=1}^n \frac{|T|}{|G:H_r|}
  <\frac{1}{2}\frac{|G|}{|T|^2}=|C|^2/2|G|\le n_a(C)$$
triples are contained in these proper subgroups. For the groups
$\OO_{10}^-(2)$, $\OO_{12}^-(2)$ in Table~\ref{tab:o-} explicit calculation
produces generating triples as claimed. For $\OO_{20}^-(2)$ the index of the
subgroup $\fA_{20}$ is too large.
\end{proof}

\begin{prop}   \label{prop:Bn}
 Theorem~\ref{thm:main} holds for the groups $\OO_{2n+1}(q)$, $n\ge3$,
 $q$ odd, with $C$ containing elements of order~$\Phi_{2n}^*(q)$. 
\end{prop}

\begin{proof}
As in the previous proof we first estimate the structure constant $n_a(C)$
in the group $G:=\SO_{2n+1}(q)$ of adjoint type. Let $x\in C$ and $T$ its
centralizers in $G$, a maximal torus of order $|T|=q^n+1$, parametrized by
$v\in W$ with cyclic centralizer of order $2n$ and $C$ the class of a
generator of $T\cap G'$.
As in the proof of the previous result we see that the extensions to $G$ of
unipotent characters of $G'$ which do not vanish on $x$ take value~$\pm1$
on $x$. Furthermore,
$|N_G(T):T|=2n$, so there are exactly $2(2n-1)$ non-linear such characters.
The minimal degree of a non-trivial unipotent character of $G$ equals
$b=(q^n-1)(q^n-q)/2(q+1)$ by \cite[Prop.~5.1]{TZ}, so the contribution of
characters from $\cE(G,s)$ with $s\in Z(G^*)$ to $|\eps_a(C)|$ is at most
$$(2n-1)\frac{1}{b}=\frac{(2n-1)2(q+1)}{(q^n-1)(q^n-q)}.$$
The non-central elements in the torus $T^*$ in the dual group
$G^*=\Sp_{2n}(q)$ have centralizer $Z_r:=\GU_{n/r}(q^r)$, with $r|n$ such
that $n/r$ is odd. Then $C_{W}(v)\cong W(T) \cong C_{2n}$, $W(s)=\fS_{n/r}$,
and $C_{W(s)}(v)\cong N_{Z_r}(T)/T \cong C_{n/r}$, so by
Proposition~\ref{prop:values} we have $|\chi_{s_r}(x)|\le 2r$. Arguing as for
${\PCO_{2n}^\circ}^-(q)$ we see that the non-unipotent characters contribute
at most
$$\sum_{r|n}\frac{n q^r (2r)^3}{r^2|G:Z_r|_{q'}}
  =\sum_{r|n}\frac{8nrq^r}{|G:Z_r|_{q'}}$$
to $|\eps_a(C)|$. Again, it ensues that $n_a(C)\ge |C|^2/2|G|$ (recall that
$q>2$ since $q$ is odd). \par
We now estimate the contribution from maximal subgroups. The maximal subgroups
containing elements of order divisible by $\Phi_{2n}^*(q)$ were described in
Lemma~\ref{lem:maxSOodd}. The stabilizers of non-degenerate 1-spaces
$H:=N_G(\SO_{2n}^-(q))$ have index $|G:H|>q^2(q^n+1)$, and for $n=3$, $G_2(q)$
has index $q^3(q^4-1)$ and $\PSU_3(q)$ has index $q^6(q^3-1)(q^4-1)$, and for
both groups there are $q+1$ distinct conjugates above $x$. Thus,
the trivial estimate that there are at most $|H|$ triples in $H$ shows that
not all triples can lie inside these proper subgroups. The Ree groups
$\tw2G_2(q)$ do not possess elements of order $\Phi_6^*(q)$ for $q\ge27$.
For $\OO_7(3)$ and $\OO_7(5)$, explicit computation yields that there exist
generating triples consisting of elements of order~7.
\end{proof}

\begin{prop}   \label{prop:Cn}
 Theorem~\ref{thm:main} holds for the groups $\PSp_{2n}(q)$, $n\ge2$,
 $(n,q)\ne(2,2)$, with $C$ containing elements of order~$\Phi_{2n}^*(q)$,
 respectively of order~5 when $(n,q)=(2,3)$.
\end{prop}

\begin{proof}
The claim for $\PSp_4(3)$ can be checked by computer, so now assume that
$(n,q)\ne(2,3)$. Then all elements $x\in C$ are regular with centralizer a
cyclic torus $T$ of order $q^n+1$ in the adjoint type group $G:=\PCSp_{2n}(q)$.
As in the previous proof, using \cite[Prop.~5.1]{TZ} we
obtain that the non-linear unipotent characters of $G'$ contribute at most
$$\gcd(2,q-1)\frac{(2n-1)2(q+1)}{(q^n-1)(q^n-q)}$$
to $|\eps_a(C)|$ in $G$. The non-trivial elements in the dual torus of order
$q^n+1$ of $G^*=\Spin_{2n+1}(q)$ have centralizers of types $\Spin_{2n}^-(q)$
(when $q^n\equiv3\pmod4$) and $\GU_{n/r}(q^r)$ for $r|n$ with $n/r$ odd. The
semisimple elements with centralizer of types $\Spin_{2n}^-(q)$ lead to the
Weil character of $G$ of degree $q^n-1$, which takes value~$-2$ on $x$. Via
Jordan decomposition the characters in this Lusztig series $\cE(G,s)$ are in
bijection with the unipotent characters of $\Spin_{2n}^-(q)$; thus the second
smallest degree in $\cE(G,s)$ equals $b|G^*:\Spin_{2n}^-(q)|_{q'}=b(q^n-1)$
with $b=(q^n+1)(q^{n-1}-q)/(q^2-1)$
by \cite[Thm.~1.1]{TZ}. Since $|\cE(G,s)|=n$ these characters contribute at most
$$\frac{1}{q^n-1}+\frac{(n-1)(q^2-1)}{(q^n+1)(q^{n-1}-q)(q^n-1)}$$
to $|\eps_a(C)|$ (and only when $q^n\equiv3\pmod4$). \par
As in the previous proof, the characters parametrized by semisimple elements
with centralizer $Z_r:=\GU_{n/r}(q)$ contribute at most
$$\sum_{r|n}\frac{8nrq^r}{|G:Z_r|_{q'}}.$$
For $n\ge3$ and $(n,q)\ne(3,2)$ this shows that $|\eps_a(C)|<0.5$. Explicit
computation of the structure constant in $\PCSp_4(q)$ with \Chevie\ shows
that here $|\eps_a(C)|<0.5$ for $q\ge3$ as well.
\par
The maximal subgroups of $G$ containing elements of order divisible by
$\Phi_{2n}^*(q)$ are given in Lemma~\ref{lem:maxSp} and Table~\ref{tab:sp}.
The index of $H_1:=N_G(\SU_n(q))$, for $n$ odd, is at least $q^5(q^n+1)$, the
index of $H_r:=N_G(\Sp_{2n/r}(q^r))$ is at least $q^{10}(q^n+1)$ for $n\ge4$,
respectively $q^2(q^2-1)$ for $n=r=2$. The index of $G_2(q)$ for $n=3$ equals
$q^3(q^4-1)$, and there are $q+1$ distinct conjugates containing $x$.
The structure constant $n_a(C)$
in $H=\SO_{2n}^-(q)$ was investigated in Proposition~\ref{prop:2Dn}; it is less
than $2|H|/|T|^2$. By explicit computation in {\sf GAP}, the groups
$\Sp_{2n}(q)$ listed in Table~\ref{tab:sp} different from $\Sp_{20}(2)$
possess generating triples of the stated form. For $\Sp_{20}(2)$ the structure
constant in the exceptional subgroup $\PSL_2(41)$ is too small. The subgroups
$\tw2B_2(q)$ of $\Sp_4(q)$ in Lemma~\ref{lem:maxSp}(5) do not possess elements
of order $\Phi_{2n}^*$ for $q\ne2$. 
\end{proof}

The excluded group $\PSp_4(2)$ has derived group the alternating group $\fA_6$
which will be considered in Lemma~\ref{small n}.

To deal with the even-dimensional split orthogonal groups we need to understand
the elements in a certain maximal torus of the dual group.

\begin{lem}   \label{lem:torusDn}
 Let $n\ge4$, $T$ a maximal torus of $G=\Spin_{2n}^+(q)$ of order
 $(q^{n-1}+1)(q+1)$. Then for $s\in T\setminus Z(G)$ either
 $C_G(s)\cong H:=\Spin_{2n-2}^-(q)(q+1)$ or $C_G(s)\le\GU_n(q)$ if $n\ge6$ is
 even, or $C_G(s)\le\GU_{n-1}(q)(q+1)$ if $n=4$ or $n$ is odd. Moreover, the
 number of elements with centralizer $H$ is given in Table~\ref{tab:centO+}.
\end{lem}

\begin{table}[htbp] 
  \caption{Centralizer $\Spin_{2n-2}^-(q)(q+1)$ in $\Spin_{2n}^+(q)$}
  \label{tab:centO+}
\[\begin{array}{|l||c|c|c|}    
\hline
 & q\text{ even}& q^n\equiv1\pmod4& q^n\equiv3\pmod4\\
\skipa \hline \hline
 n=4& 3q/2& 3(q-1)& -\\
 n>4& q/2& q-1& q\\    \hline
\end{array}\]
\end{table}

\begin{proof}
Assume that $q$ is odd. We first work inside the group $\SO_{2n}^+(q)$.
Under the action of $T$, the natural module decomposes into an orthogonal sum
$V_1\perp V_2$, with $\dim V_1=2n-2$, $\dim V_2=2$. The elements $x\in T$
for which $x|_{V_1}$ and $x|_{V_2}$ do not have the same eigenvalues have
centralizer contained in $\SO_{2n-2}^-(q)\times \SO_2^-(q)$. Moreover,
if $x|_{V_1}\ne\pm\Id$ then the centralizer is contained in
$\GU_{n-1}(q)\times(q+1)$. Thus, elements with centralizer
$\SO_{2n-2}^-(q)\times \SO_2^-(q)$ are those acting as $\pm\Id$ on $V_1$ (and
of order dividing $q+1$ on $V_2$), hence they are the elements in a subgroup
$H\cong C_2\times C_{q+1}$ of $T$, and they are all real. Apart from the two
central elements, we have $2q$ such elements and $q$ classes. Now let's study
the elements which lie inside $H':=H\cap \Omega_{2n}^+(q)$.
If $q\equiv1\pmod4$, then $H'$ is cyclic of order $q+1$ with $q-1$ noncentral
elements in $(q-1)/2$ classes. Lifting to $\Spin_{2n}^+(q)$ each element has
two preimages, so we get $q-1$ classes.
\par
If $q\equiv3\pmod4$ and $n$ is even, $H'\cong C_2\times C_{(q +1)/2}$ which
contains 3 involutions (1 central),  so there are $q-1$ noncentral elements
including two involutions, lying in $(q-3)/2+2$ classes. The involutions lift
to one class each, all other classes have two preimage classes, so there are
again $(q-3)+2=q-1$ classes.
\par
If $q\equiv3\pmod4$ and $n$ is odd, then the central involution has spinor
norm -1 and $H'$ has only one involution, so it is cyclic of order $q+1$ with
no nontrivial central elements. Thus we find $q$ elements with this
centralizer and so $(q-1)/2+1=(q+1)/2$ classes in $\Omega_{2n}^+(q)$. Again,
the only element with only one preimage class is the involution, so we find
$q$ classes in $G$.
\par
The elements $x\in T$ for which $x|_{V_1}$ and $x|_{V_2}$ have the same
eigenvalues (which can only happen if $n$ is even), have centralizer
$\GU_n(q)$. For $n=4$, their lifts to $\Spin_8^+(q)$ are interchanged with
those of centralizer $\Spin_{2n-2}^-(q)(q+1)$ by triality, so we get
$3(q-1)$ classes with this centralizer in total. \par
The case where $q$ is even is much easier since there
$\Spin_{2n}^+(q)=\SO_{2n}^+(q)$.
\end{proof}

\begin{prop}   \label{prop:Dn}
 Theorem~\ref{thm:main} holds for the groups $\OO_{2n}^+(q)$, $n\ge4$,
 with $C$ containing elements which acts as an element of order
 $\Phi_{2n-2}^*(q)$ on a $2n-2$-dimensional subspace and as an element of
 order $>2$ dividing~$q+1$ on its orthogonal complement. 
\end{prop}

\begin{proof}
Let $C$ be the conjugacy class of an element $x\in \OO_{2n}^+(q)$ as in the
statement. Then $x$ is regular semisimple in a maximal torus $T$ of order
$(q^{n-1}+1)(q+1)$ in the group $G:={\PCO_{2n}^\circ}^+(q)$ of adjoint
type. Such an element exists unless $(n,q)=(4,2)$. For the group $\OO_8^+(2)$
it can be checked using {\sf GAP} that there exist triples as required with
elements of order~9, so now assume that $(n,q)\ne(4,2)$.
We estimate the structure constant $n_a(C)$ in $G$. The characters of $G$
not vanishing on $x$ lie in Lusztig series parametrized by semisimple elements
in the dual torus $T^*\le G^*=\Spin_{2n}^+(q)$. We distinguish three types of
such elements: the central elements, the elements with centralizer
$\Spin_{2n-2}^-(q)\times\Spin_2^-(q)$, and the remaining elements, having
centralizer $\GU_n(q)$ (for $n$ even) or smaller, see Lemma~\ref{lem:torusDn}.
\par
The Lusztig series for central elements $s\in T^*$ contain the extensions to
$G$ of unipotent characters of $G'$. Those which do not vanish on $x$ take
value~$\pm1$ on $x$. Furthermore, $|N_G(T):T|=2n-2$, so there are exactly
$d(2n-3)$ non-linear such characters. The minimal degree of a non-linear
unipotent character of $G$ equals $b_1:=(q^n-1)(q^{n-1}+q)/(q^2-1)$ by
\cite[Prop.~7.2]{TZ}, so the contribution of characters from $\cE(G,s)$
with $s\in Z(G^*)$ to $|\eps_a(C)|$ is at most
$$(2n-3)\frac{1}{b_1}=\frac{(2n-3)(q^2-1)}{(q^n-1)(q^{n-1}+q)}.$$
Next, let $s\in T^*$ with centralizer $\Spin_{2n-2}^-(q)\Spin_2^-(q)$;
the corresponding semisimple character $\chi_s$ has degree
$b_2:=(q^n-1)(q^{n-1}-1)/(q+1)$. By Proposition~\ref{prop:values}
$\chi_s(x^a)$ has absolute value at most~2, and thus the same holds for all
the $n-1$ characters in $\cE(G,s)$. By Lemma~\ref{lem:torusDn} there are at
most $q$ such semisimple elements up to conjugation, so these characters
contribute at most
$$q(n-1)2^3\frac{1}{b_2}=\frac{8(n-1)q(q+1)}{(q^n-1)(q^{n-1}-1)}$$
to $\eps_a(G)$. If $n$ is even, there are at most $(q+1)^2-2(q+1)=q^2-1$
further elements $s_0$ with centralizer $Z_0:=\GU_n(q)$, each conjugate to its
inverse. The remaining elements in $T^*$ have centralizer
$Z_r:=\GU_{(n-1)/r}(q^r)(q+1)$, with $r|(n-1)$ such that $(n-1)/r$ is odd,
and for each $r$ there are at most $(q^r+1)(q+1)-(q+1)^2=(q^r-q)(q+1)<q^r(q+1)$
such elements $s_r$, falling into at most $q^r(q+1)/r$ classes.
By Proposition~\ref{prop:values} we have $|\chi_{s_0}(x)|\le 2$, respectively
$|\chi_{s_r}(x)|\le 2r$. Thus, the remaining characters contribute at most
$$\frac{8(n-1)(q^2-1)}{|G:Z_0|_{q'}}+
   \sum_{r|n-1}\frac{(n-1)q^r(q+1)(2r)^3}{r^2|G:Z_r|_{q'}}$$
to $|\eps_a(C)|$. This shows that $n_a(C)\ge \frac{1}{2}|C|^2/|G|$ for
$(n,q)\ne(4,2)$.
\par
The maximal subgroups of $G$ containing elements of order divisible by
$\Phi_{2n-2}^*(q)$ are given in Lemma~\ref{lem:maxSO+} and Table~\ref{tab:o+}.
By our choice of $x$ it stabilizes a unique decomposition of the underlying
space, into a 2-dimensional subspace and its orthogonal complement. The
stabilizer $H_1:=\GO_{2n-2}^-(q)\times \GO_{2}^-(q)\cap G$ of this
decomposition has index $q^{2(n-1)}(q^n-1)(q^{n-1}-1)/(q+1)$.
The index of $H_2:=N_G(\SU_n(q))$, for $n$ even, is at least
$q^{n(n-1)/2}(q^{n-1}-1)$, the index of $H_3:=N_G(\Omega_n(q^2))$, for $n$
odd, is at least $q^{(n^2-1)/2}(q^n-1)$.
\par
For $n=4$, there is one more subgroup $\Spin_7(q)$, again with too few triples
by Proposition~\ref{prop:Bn}, and a subgroup $N_G(\PSU_3(q))$, of index
at least $q^9(q^4-1)(q^3-1)$. Thus, not all triples can lie inside these
proper subgroups. By explicit computation in {\sf GAP}, the groups
$\OO_8^+(3)$ and $\OO_8^+(5)$ possesses generating triples of elements of
order~7, resp.~21.
The only other group in Table~\ref{tab:o+} containing elements of order $o(x)$
is $\fA_{16}$ in $\Omega_{14}^+(2)$, but there the number of triples is too
small.
\end{proof}

\begin{prop}   \label{prop:SUnodd}
 Theorem~\ref{thm:main} holds for the groups $\PSU_n(q)$, $n\ge3$ odd,
 $(n,q)\ne(3,2)$, with $C$ containing elements of order~$\Phi_{2n}^*(q)$.
\end{prop}

\begin{proof}
Let $x\in C$, with centralizer a cyclic torus $T$ of order $(q^n+1)/(q+1)$ in
$G:=\PGU_n(q)$. The minimal degree of a non-principal unipotent character of
$G$ is at least $(q^n-q)/(q+1)$ by \cite[Thm.~4.1]{TZ} so the non-linear
unipotent characters of $G'$ contribute at most
$$\frac{(n-1)(q+1)}{(q^n-q)}$$
to $|\eps_a(C)|$ in $G$. The non-trivial elements in the dual torus $T^*$ in
$G^*=\SU_n(q)$ of order $(q^n+1)/(q+1)$ have centralizers of types
$Z_r:=G^*\cap\GU_{n/r}(q^r)$ for $r|n$, $r>1$. Let $s_r$ denote an element with
centralizer $Z_r$. Then $C_{W}(v)\cong W(T) \cong C_n$,
$W(s)=\fS_{n/r}$, and $C_{W(s)}(v)\cong N_{Z_r}(T)/T \cong C_{n/r}$, so by
Proposition~\ref{prop:values} we have $|\chi_{s_r}(x)|\le r$. By
\cite[Thm.~7A]{FoSr1} the $n/r$ characters in $\cE(G,s_r)$ not vanishing on $x$
lie in the same $p$-block as $\chi_{s_r}$, for all Zsigmondy prime
divisors $p$ of $\Phi_{2n}(q)$, so their value on $x$ has the same absolute
value as $\chi_{s_r}(x)$. Up to conjugation there are at
most $q^{r-1}/r$ elements with centralizer $\SU_{n/r}(q^r)$, so the characters
in Lusztig series for non-central elements of $G^*$ contribute at most
$$\sum_{r|n}\frac{n q^{r-1} r^3}{r^2|G:Z_r|_{q'}}
  =\sum_{r|n}\frac{nrq^{r-1}}{|G:Z_r|_{q'}}$$
to $|\eps_a(C)|$. Using these estimates an easy calculation now shows that
$|\eps_a(C)|<0.5$ for all $n\ge5$, $(n,q)\ne(5,2)$, so that
$n_a(C)\ge |C|^2/2|G|$.
For $n=3$, the structure constant can be computed explicitly as
$|T|(q^2+3q+1)$ for $q\ne3,5$. 
\par
For the groups $\PSU_3(3)$, $\PSU_3(5)$ and $\PSU_5(2)$ it can be checked by
direct calculation that there exist generating triples. For the remaining
cases, by Lemma~\ref{lem:maxSUodd} the only maximal subgroups to consider are
the images in $\PGU_n(q)$ of normalizers of $\GU_{n/r}(q^r)$ in $\GU_n(q)$,
for $r|n$, of index at least $q^{n^2/4}(q^{n-1}-1)/r$, resp.\ $q^3(q^2-1)/3$
if $n=3$. Again, not all triples can lie in these subgroups.
\end{proof}

The excluded group $\PSU_3(2)$ is solvable.

\begin{prop}   \label{prop:SUneven}
 Theorem~\ref{thm:main} holds for the groups $\PSU_n(q)$, $n\ge4$ even,
 $(n,q)\ne(4,2)$, with $C$ containing elements of order~$\Phi_{2n-2}^*(q)$. 
\end{prop}

\begin{proof}
Let $x\in C$, with centralizer a cyclic torus $T$ of order $q^{n-1}+1$ in
$G:=\PGU_n(q)$. The minimal degree of a non-principal unipotent character of
$G'$ is at least $(q^n+q)/(q+1)$ by \cite[Thm.~4.1]{TZ}, so the $n-2$
non-linear unipotent characters of $G'$ not vanishing on regular elements
of $T$ contribute at most
$$\frac{(n-2)(q+1)}{(q^n+q)}$$
to $|\eps_a(C)|$ in $G$. The non-trivial elements in the dual torus $T^*$ in
$G^*=\SU_n(q)$ of the same order have centralizers of types
$Z_r:=\GU_{(n-1)/r}(q^r)$ for $r|(n-1)$. Let $s_r$ denote an element with
centralizer $\GU_{(n-1)/r}(q^r)$. Then $C_{W}(v)\cong W(T) \cong C_{n-1}$,
$W(s)=\fS_{(n-1)/r}$, and $C_{W(s)}(v)\cong N_{Z_r}(T)/T \cong C_{(n-1)/r}$,
so by Proposition~\ref{prop:values} we have $|\chi_{s_r}(x)|\le r$. By
\cite[Thm.~7A]{FoSr1} the $(n-1)/r$ characters in $\cE(G,s_r)$ not vanishing
on $x$ lie in the same $p$-block as $\chi_{s_r}$, for all Zsigmondy prime
divisors $p$ of $\Phi_{2n-2}(q)$, so their value on $x$ has the same absolute
value as $\chi_{s_r}(x)$. Up to conjugation there are at
most $q^{r-1}/r$ elements with centralizer $\GU_{(n-1)/r}(q^r)$, so the
characters in Lusztig series for non-central elements of $G^*$ contribute at
most
$$\sum_{r|n}\frac{(n-1) q^{r-1} r^3}{r^2|G:Z_r|_{q'}}
  =\sum_{r|n}\frac{(n-1)rq^{r-1}}{|G:Z_r|_{q'}}$$
to $|\eps_a(C)|$. It follows that $|\eps_a(C)|<0.5$ and thus
$n_a(C)\ge |C|^2/2|G|$ for all even $n\ge4$, $(n,q)\ne(4,2),(4,3),(6,2)$.
\par
For $\PSU_6(2)$, $\PSU_4(3)$ and $\PSU_4(5)$ direct calculation shows that
there exist generating triples. By Lemma~\ref{lem:maxSUeven} the only maximal
subgroup to consider in the remaining cases is $\GU_{n-1}(q)$, of index
$q^{n-1}(q^n-1)/(q+1)$, thus not all triples can lie in proper subgroups.
\end{proof}

The excluded group $\PSU_4(2)$ is isomorphic to $\PSp_4(3)$, treated in
Proposition~\ref{prop:Cn}.

\begin{prop}   \label{prop:SLn}
 Theorem~\ref{thm:main} holds for the groups $\PSL_n(q)$, $n\ge3$,
 $(n,q)\ne(3,2),(4,2)$, with
 \begin{enumerate}
  \item[\rm(a)] $C$ containing elements of order~$\Phi_n^*(q)$ if $n$ is
   odd, and
  \item[\rm(b)] $C$ containing elements of order~$\Phi_{n-1}^*(q)$, if $n$ is
   even.
 \end{enumerate}
\end{prop}

\begin{proof}
First assume that $n$ is odd.
Let $x\in C$, with centralizer a cyclic torus $T$ of order $(q^n-1)/(q-1)$ in
$G:=\PGL_n(q)$. The minimal degree of a non-principal unipotent character of
$G$ is at least $(q^n-q)/(q-1)$ by \cite[Thm.~3.1]{TZ} so the non-linear
unipotent characters of $G'$ contribute at most
$$\frac{(n-1)(q-1)}{(q^n-q)}$$
to $|\eps_a(C)|$ in $G$. The non-central elements in the dual torus $T^*$ in
$G^*=\SL_n(q)$ of the same order have centralizers of types
$Z_r:=G^*\cap\GL_{n/r}(q^r)$ for $r|n$, $r>1$. Let $s_r$ denote an element with
centralizer $Z_r$. Then $C_{W}(v)\cong W(T) \cong C_n$,
$W(s)=\fS_{n/r}$, and $C_{W(s)}(v)\cong N_{Z_r}(T)/T \cong C_{n/r}$, so by
Proposition~\ref{prop:values} we have $|\chi_{s_r}(x)|\le r$. Again by
\cite[Thm.~7A]{FoSr1} the values of the $n/r$ characters in $\cE(G,s_r)$ not
vanishing on $x$ have the same absolute value as $\chi_{s_r}(x)$. Up to
conjugation there are at most $q^{r-1}/r$ elements with centralizer $Z_r$,
so the characters in Lusztig series for non-central elements of $G^*$
contribute at most
$$\sum_{r|n}\frac{n q^{r-1} r^3}{r^2|G:Z_r|_{q'}}
  =\sum_{r|n}\frac{nrq^{r-1}}{|G:Z_r|_{q'}}$$
to $|\eps_a(C)|$. It follows that $|\eps_a(C)|<0.5$ for all $n\ge5$,
$(n,q)\ne(5,2)$, so that $n_a(C)\ge \frac{1}{2}|C|^2/|G|$.
For $n=3$, the structure constant can be computed explicitly from the known
character table as $|T|(q^2-3q+1)$ for $q\ne2,4$. 
\par
By direct computation, the groups $\PSL_3(4)$ and $\PSL_5(2)$
contain generating triples. By Lemma~\ref{lem:maxSLn}, the only maximal
subgroups to consider in the remaining cases are the image in $G$ of
$\GU_n(q^{1/2})$ if $q$ is a square, of index at least
$q^{n(n-1)/4}(q^{n-1}-1)$ and with $q^{n/2}-1$ conjugates containing $x$,
and of the normalizers of $\GL_{n/r}(q^r)$ in
$\GL_n(q)$, for $r|n$, $r>1$, of index at least $q^{n^2/4}(q^{n-1}-1)/r$,
resp.\ $q^3(q^2-1)/3$ if $n=3$. \par

If now $n\ge4$ is even, the same arguments apply to show that $|\eps_a(C)<0.5|$
unless $(n,q)=(4,2)$, so that $n_a(C)\ge \frac{1}{2}|C|^2/|G|$. \par
By direct computation, the group $\PSL_4(4)$ contains generating
triples. The only maximal subgroups to consider in the remaining cases are,
by Lemma~\ref{lem:maxSLn}, the image in $G$ of $\GU_n(q^{1/2})$, with $q$ a
square, of large index and less than $q^{n/2}$ conjugates containing $x$, and
the end node
parabolic subgroups $P_1$, $P_{n-1}$, of index $(q^n-1)/(q-1)$. Since both
$P_i$ contain the full normalizer of the maximal torus $T$ of order $q^{n-1}-1$,
there are at most $|P_i|/|T|$ elements from $C$ in $P_i$, so each contains
at most $(q-1)|G|/((q^n-1)(q^{n-1}-1))$ triples, which is sufficiently small
for our estimate.
\end{proof}

The excluded group $\PSL_4(2)\cong\fA_8$ will be treated in Lemma~\ref{small n}.
The groups $\PSL_2(q)$ are the subject of the next section.

\subsection{The groups $\PSL_2(q)$}   \label{subsec:L2q}
We now consider $\SL_2(q)$ and $\PSL_2(q)$ where the results are somewhat
different especially for $q$ even. 
    
From the subgroup structure of $\SL_2(q)$, it is quite easy to see that if
$q > 11$, then the only maximal subgroups containing the split torus are the
normalizer of the torus and the two Borel subgroups.  Similarly, the
normalizer of the nonsplit torus is the unique maximal subgroup containing
the nonsplit torus.

We first deal with $q$ odd.

\begin{lem}   \label{sl2q odd}
 Let $S=\PSL_2(q)$ with $q \ge11$ and odd. Let $C$ be a conjugacy class
 of elements of order $(q-1)/2$. Then Theorem~\ref{thm:main} holds for $C$.
 Also, Theorem~\ref{thm:main} holds for $\PSL_2(9)$ with $C$ a class of
 elements of order~5.
\end{lem}

\begin{proof}
We work in $G=\SL_2(q)$. First assume that $q>11$. Let $C$ be a conjugacy class
of elements of order $q-1$ in $G$. It is a trivial matrix computation to
show that $\tr(xy)$ with $x,y \in C$ can be any element of $\FF_q$. Thus,
there exist $x,y \in C$ and $z \in C$ or $z \in -C^{-2}$ so that $xyz=1$.
We claim that $G=\langle x,y \rangle$ in either case. If not, then $x,y$ are
contained in a Borel subgroup $B$. However, if $x,y \in C \cap B$, we see
that $xy$ is either unipotent or conjugate to $x^2$. Thus,
$G=\langle x, y \rangle$. Passing to the quotient $\PSL_2(q)$ gives the
result. For $q=9,11$ the claim follows by direct computation.
\end{proof}

If $G=\PSL_2(7)$, then it is straightforward to compute that if $C$
is a conjugacy class of elements of order $7$, there exist $x,y \in C$
with $xy \in C^{-2}$ with $G=\langle x, y \rangle$.   One easily checks
that $G$ is a true exception to the first statement of Theorem \ref{thm:main}.

If $G=\PSL_2(5) \cong \PSL_2(4) \cong \fA_5$, then there is a generating
triple of conjugate elements of order $5$ which generate $G$.  
It is also straightforward to check that if $V$ is a non-trivial
absolutely irreducible $kG$-module with $k$ of characteristic
different than $2$, then an element of order $5$ has no
eigenspace of dimension greater than $(1/3) \dim V$.  If $k$ has
characteristic $2$, then any element of order $5$ has a trivial
fixed space. 

We next consider $G=\PSL_2(q)$ with  $q$ even.  If $q=4$, then 
$G=\fA_5 \cong \PSL_2(5)$, a case just dealt with.   

\begin{lem}  \label{sl2q even}
 Let $G=\PSL_2(q)$ with $q > 7$ and even. Let $C$ be a conjugacy class of
 elements of order $q-1$. Let $k$ be an algebraically closed field of
 characteristic $p$. 
 \begin{enumerate}
  \item[\rm(a)] There exist $x,y,z\in C$ with product $1$ and
   $G=\langle x, y\rangle$.
  \item[\rm(b)] If $p > 2$ and $V$ is any nontrivial irreducible $kG$-module,
   then every eigenspace of $x$ has dimension at most $(1/3) \dim V$.
  \item[\rm(c)] If $p =2$ and $V$ is any  nontrivial irreducible $kG$-module of
   dimension greater than $2$, then every eigenspace of $x$ has dimension at
   most $(1/3) \dim V$.
 \end{enumerate}
\end{lem}

\begin{proof}
(a) follows just as for $q$ odd. 
Now first suppose that $p > 2$. Thus, $\dim V = q - 1 +e$ where $e = 0,1,2$
by the well-known representation theory of $\PSL_2(q)=\PGL_2(q)$. Let $B$ be
a Borel subgroup containing $x \in C$. Let $U$ be the unipotent radical of $B$.
Then $U$ has $q-1$ nontrivial eigenspaces on $V$ permuted transitively by
$x$. Thus $V$ is a direct
sum of a rank $1$  free $x$-module and an $x$-submodule of dimension $e$.
Thus, any eigenspace of $x$ has dimension at most $1 + e$, whence (b) holds. 
 
Now assume that $p=2$.  Then every nontrivial irreducible module is a tensor
product of Frobenius twists of the natural module.  It is trivial to check
that either every eigenspace for elements of $C$ has dimension at most $1$
or $V$ is the Steinberg module of dimension $q$ and the largest eigenspace
is the trivial eigenspace of dimension $2 < q/3$.
\end{proof}

We close this section with a result for the nonsplit tori in $\SL_2(q)$ with 
$q$ even.

\begin{lem}  \label{sl2q even2}
 Let $G=\SL_2(q)$ with $q  > 3$ even.  Let $x \in G$ have order~$q+1$.
 If $k$ is an algebraically closed field of characteristic $2$ and $V$ is an
 irreducible nontrivial $kG$-module, then $x$ has distinct eigenvalues on
 $V$ and $C_V(x)=0$.
\end{lem}

\begin{proof}
$V$ is a tensor product of distinct Frobenius twists of the natural
$2$-dimensional module. If $q=2^f$, there are exactly $f$ distinct twists.
It is straightforward to see that possible eigenvalues are distinct on $V$
and $C_V(x)=0$.
\end{proof}

\section{Alternating and sporadic groups}   \label{sec:spor}
In this section we prove Theorem~\ref{thm:main} for alternating and sporadic
groups. We first recall a translation result for generation. See
\cite[Lemma~4.6]{gurneu} for example. 

\begin{lem} \label{translation}
 Let $G$ be a finite group generated by $x,y,z$ with product $1$. Then for
 any $d\ge1$ there is a normal subgroup $N$ generated by $x^d, y^d$ and $d$
 conjugates of $z$ with product $1$ generating $N$. Moreover $G/N$ is cyclic
 of order dividing $d$. 
\end{lem}

\begin{lem}  \label{n-2 cycles}
 Let $n\ge11$ be odd and set $G=\fA_n$.  Then there exist three $n-2$
 cycles with product $1$ that generate $G$. 
\end{lem}

\begin{proof}
Set  $x = ( 2 \ 4 \ 5 \ 6 \ldots n)$ and $u=(1 \ 2)(3 \ 4)$.
Then $w:=ux = (1 \ 2 \  \ldots n)$.  Set $s=w^4$.   
Then $v:=u^s=(5 \ 6)(7 \ 8)$.   Set $y = x^s$.  So $w=vy$.
Thus,  $y^{-1}(vu)x=1$ with $vu$ an involution moving $8$ points.
Let $H: = \langle x,y \rangle$.    Since $y$ does not fix
$1$ or $3$ or $\{1, 3 \}$,   $H$ is transitive.  Since $x$ is an 
$n-2$ cycle and $n$ is odd,  $H$ is primitive and then
applying \cite[Thm. 13.8]{Wie}, $H$ is triply transitive.
It follows by \cite[Thm. 15.1]{Wie} that either $H=G$
or $n \le 21$.  By inspection of the triply transitive degree $n$
groups with $11\le n\le 21$ generated by $n-2$ cycles, no possibility
remains for $H$ since $n\ge11$.
 
So $G=\langle x,y \rangle$ where $x$ and $y$ are $n-2$ cycles
and their product is an involution.  By the translation principle,
a subgroup of index at most $2$ in $G$ is generated by $x^2$
and two conjugates of $y$ with product $1$.
\end{proof}  
 
\begin{lem} \label{n-3 cycles}
 Let $n\ge12$ be even and set $G=\fA_n$. Then there exist three $n-3$
 cycles with product $1$ that generate $G$.
\end{lem}

\begin{proof}
Set $u =(1 \ 2)(3 \ 4)(5 \ 6)$ and $x = ( 1 \ 3 \ 5 \ 7 \ 8  \ldots n)$.
Then $w = xu$ is the standard $n$-cycle.  Set $s=w^6$.
So $w=yv$ where $y=x^s$ and $v=u^s$.  Note that $vu$ is an involution moving
exactly $12$ points and $y^{-1}(vu)x=1$.  Set $H= \langle x,y \rangle$.
By construction $H$ is transitive.   It is obviously primitive unless $3|n$
and then the fixed points $F$ of $x$ would have to be a block (i.e. $\{2,4,6\}$
as well as $vu(F)=\{1,3,5\}$).  Since $x(1)=3$ and $x(5) =7$, this is
a contradiction. So $H$ is primitive.

By \cite[Thm. 13.8]{Wie}, $H$ is $4$-transitive.  
By \cite[Thm. 15.1]{Wie}, $n \le 25$.   The only possibilities
for $H$ other than $G$ would for $H$ to be a Mathieu group with
$n=12, 22$ or $24$.  In all three cases, we see that the Mathieu group $M_n$
does not contain both an involution moving exactly $12$ points and an
$n-3$ cycle.
\end{proof}

A computer check shows the following:

\begin{lem} \label{small n}
 Let $5\le n\le 10$. Then Theorem~\ref{thm:main} holds for $\fA_n$ with $C$ the
 class of a $k$-cycle, where $k=5$ for $n=5,6$ and $k=7$ otherwise.
\end{lem}

The result is pretty close for $\fA_{10}$: for a fixed element $x$ of order~7,
only 42 out of the 7446 pairs $(y,z)$ with $y,z$ conjugate to $x$ and $xyz=1$
do generate $G$ (this is roughly one in~177), all others generate intransitive
subgroups: for example, 2856 pairs generate an $\fA_9$, 3717 generate an
$\fA_8$. The generating triple for $\fA_{10}$ can be obtained by translation
from the rigid genus~0 triple of $\fS_{10}$ consisting of elements of cycle
shapes $(2^5,7,7.2)$.

\begin{prop}   \label{prop:spor}
 Theorem~\ref{thm:main} holds for sporadic groups and for the Tits group,
 with $C$ as indicated in Table~\ref{tab:sporadic}.
\end{prop}

\begin{table}[htbp] 
  \caption{Structure constants for sporadic groups} \label{tab:sporadic}
\[\begin{array}{|l|rcr|l|}    
\hline
 G& C& A_G(C)& n_a(C)& \text{max. overgroups of $g\in C$}\\
\hline\hline
 M_{11}& 11a&  5& 35|80& \PSL_2(11):2|14\\
 M_{12}& 11a&  5& 640|1180& M_{11}:35|80\ (2\times),\PSL_2(11):2|14\\
    J_1& 19a&  6& 496|419& 19.6:3|3\\
 M_{22}& 11a&  5& 3632|3776& \PSL_2(11):2|14\\
    J_2&  7a&  6& 12528& \PSU_3(3):397,\PSL_3(2).2:12\\
 M_{23}& 23a& 11& 17646|18222& 23.11:11\\
\tw2F_4(2)'& 13a& 6& 114870|114195& \PSL_3(3).2: 106|133 (2\times),\PSL_2(25):1650\\
     HS& 11a&  5& 363464|367964& M_{22}:3632|3776, M_{11}:35|80\ (2\times)\\
    J_3& 19a&  9& 131161& \PSL_2(19):4\ (2\times)\\
 M_{24}& 23a& 11& 441904|455728& M_{23}:17646|18222,\PSL_2(23):5|29\\
    McL& 11a&  5& 7372675|7463800& M_{22}:3632|3776\ (2\times),M_{11}:35|80\\
     He& 17a&  8& 14113998& \PSp_4(4).2:13892\\
     Ru& 29a& 14& 174426828& \PSL_2(29):35\\
    Suz& 13a&  6& *& G_2(4),\PSL_3(3).2\ (2\times),\PSL_2(25)\\
     ON& 31a& 15& 479254117& \PSL_2(31):7\ (2\times)\\
   Co_3& 23a& 11& *& M_{23}\\
   Co_2& 23a& 11& *& M_{23}\\
Fi_{22}& 13a&  6& *& \OO_7(3)\ (2\times), \tw2F_4(2)'\\
     HN& 19a&  9& 756228015580& \PSU_3(8).3:134923\\
     Ly& 67a& 22& *& 67.22\\
     Th& 19a& 18& 252411100157582& \PSU_3(8).6:539180,\PSL_2(19).2:36\\
Fi_{23}& 17a& 16& *& \PSp_8(2),\PSp_4(4).4\\       
   Co_1& 13a& 12& *& 3.Suz.2, (\fA_4\times G_2(4)).2\\
    J_4& 43a& 14& *& 43.14\\
Fi_{24}'& 29a&14& *& 29.14\\
      B& 47a& 23& *& 47:23\\            
      M& 71a& 35& *& \PSL_2(71)\\       
\hline
\end{array}\]
'*': the values are too large to be printed\hskip 3cm
\end{table}

\begin{proof}
In Table~\ref{tab:sporadic} we give, for each sporadic simple group $G$, a
conjugacy class $C$ (in Atlas notation), the index
$A_G(C):=|N_G(\langle g\rangle):C_G(g)|$ for $g\in C$, the structure constant
$n_1(C)$ (and, if that is different, $n_{-2}(C)$) in $G$, and the list
of conjugacy classes of maximal subgroups
containing an element from $C$. Here, the structure constant in $G$
as well as in any maximal subgroup $H$ non-trivially intersecting $C$ is
easily computed from the character tables, using {\sf GAP} for example, since
we chose $C$ such that only almost quasi-simple or metabelian maximal
subgroups $H$ occur.
It remains to check that these contributions are less than $n_a(C)$. \par
The lists of maximal subgroups are taken from the Atlas \cite{Atl} or from
the Atlas home page \cite{webAtl}.
\end{proof}
    
\section{An application to representation theory} \label{peterconj}

Here, we prove our main results Corollary~\ref{cor:main} and
Theorem~\ref{thm:main2} from the introduction.

We first recall Scott's Lemma \cite{Sc}.
Let $k$ be a  field of characteristic $p \ge 0$.   

\begin{lem} \label{Scott}
 Suppose that $G = \langle g_1, \ldots, g_r \rangle$ with $g_1\cdots g_r=1$.
 Then for any finite dimensional $kG$-module $V$ we have
 $$\sum_{i=1}^r \dim [g_i,V] \ge \dim V - \dim V^G +  \dim [G,V].$$
\end{lem}

We shall apply this when $r=3$ and $G$ has no fixed points on $V$ or $V^*$.
Noting that $\dim C_V(g_i) = \dim V - \dim [g_i,V]$, this gives
$$
\sum_{i=1}^3 \dim C_V(g_i) \le \dim V.
$$

Recall our notation $C^a:=\{x^a\mid x \in C\}$ for $C$ a conjugacy class of
$G$. The connection between our previous results on generation and the
size of eigenspaces is made by:

\begin{lem}   \label{lem:dim1/3}
 Let $k$ be algebraically closed. Let $x,y\in\GL_n(k)=\GL(V)$ be conjugate
 with product $xy$ conjugate to $x^2$, where $n\ge2$. Set
 $G=\langle x, y \rangle$. If $V$ contains no $1$-dimensional $kG$-submodules
 and has no $1$-dimensional $kG$-quotient module, then every eigenspace of
 $x$ has dimension at most $n/3$.
\end{lem}

\begin{proof}
Let $\theta$ be an eigenvalue of $x$ with the $\theta$-eigenspace of maximal
dimension among all eigenspaces. Set $z=(xy)^{-1}$,
$$x'=\theta^{-1}x,\quad y'=\theta^{-1}y,\quad z'=\theta^2z.$$
Then $x'y'z'=1$. By hypothesis, $V$ and $V^*$ have
no fixed points for $H=\langle x',y' \rangle$.  
Note that the fixed space of each of these elements has dimension at least
that of the $\theta$-eigenspace of $x$, so an application of Scott's Lemma to
this triple shows that this dimension is at most $n/3$.
\end{proof}

Let's say that a finite group $G$ \emph{has property (E)} if there exists
$g\in G$ such that for any algebraically closed field $k$ and any $kG$-module
$V$ for which $G$ has no fixed points on $V$ or $V^*$, every eigenspace of~$g$
on $V$ has dimension at most $(1/3) \dim V$. Note that
non-trivial irreducible representations of a group $G$ with property (E)
are necessarily of dimension at least~3; in particular,
$G$ is perfect. We will say that $G$ has property (E) \emph{in characteristic
$p$} if the result holds for modules over fields of characteristic $p$.

\begin{cor}   \label{cor:propE}
 Let $G$ be a finite nonabelian simple group other than $\PSL_2(q)$ with $q$
 even.  Then $G$ has property (E). All finite nonabelian simple groups have
 property (E) in characteristic not $2$.  
\end{cor}

\begin{proof}
Let $x,y\in G$ as in Theorem~\ref{thm:main}(b). Then $G$ has property (E) with
respect to $g:=x$. The result follows by the previous Lemma unless
$G=\PSL_2(q)$ with $q$ even or $q=7$.  If $q=7$ or we are considering modules
in any characteristic other than $2$ and $q$ is even, the result follows for
irreducible modules by Lemma~\ref{sl2q even} (and the remarks above it).
For any eigenvalue other than~$1$, property (E) can be checked on irreducible
modules. The condition for the eigenvalue~$1$ follows by Scott's Lemma from
Theorem~\ref{thm:main}.
\end{proof}

The same proof in the remaining cases yields:

\begin{cor} \label{cor:weakpropE}    
 Let $G$ be a finite nonabelian simple group and $k$ an algebraically closed
 field of characteristic $p$. There exists $x \in G$ such  that if $V$
 contains no composition factors of dimension at most $2$, then every
 eigenspace of $x$ has dimension at most $(1/3) \dim V$.
\end{cor}

We now move towards composite groups:

\begin{lem}   \label{lem:prod}
 Let $G=G_1\times G_2$ be a direct product of finite groups. Assume that
 $G_1$ and $G_2$ both have property (E). Then so does $G$.
\end{lem}

\begin{proof}
We may assume that $k$ is algebraically closed. Let $g_i\in G_i$ as in
property~(E). We claim that $g:=(g_1,g_2)\in G$ satisfies (E) for $G$.
As in the previous proof we may assume that $V$ is irreducible. Thus,
$V=V_1 \otimes V_2$, with $V_i$ irreducible for $G_i$ and at least one of
them not the trivial module, say $V_1$.

By assumption every eigenspace of $g_1$ on $V_1$ has dimension at most 
$(1/3) \dim V_1$. Choose a $g_2$-invariant filtration $0<W_1<\ldots<W_r=V_2$
of $V_2$ with 1-dimensional quotients, so $r=\dim V_2$. It is clear that
on $V_1\otimes (W_{j+1}/W_j)$ each eigenspace of $g$ has the same dimension
as an eigenspace of $g_1$ and so of dimension at most $(1/3) \dim V_1$.
The result follows.
\end{proof}

The previous result implies that direct products of finite nonabelian
simple groups have property (E) in any characteristic not $2$, and also in
characteristic $2$ as long as we avoid $\PSL_2(q)$ with $q$ even.     

In this last case, we will need a different result. If $S$ is a finite subset
of $\GL(V)$, let $\avg(S,V): = |S|^{-1} \sum_{s \in S} \dim C_V(s)$.
 
\begin{cor}
 Let $G=L_1\times\ldots\times L_t$ where $L_i \cong\PSL_2(q)$ with $q\ge 4$
 even. Let $k$ be an algebraically closed field of characteristic $2$.
 Let $V$ be a $kG$-module with no trivial composition factors. Let 
 $X=X_1 \times \ldots \times X_t$ where $X_i \le L_i$ is cyclic of order $q+1$.
 \begin{enumerate}
  \item[\rm(a)]  $\avg(X,V)  \le (1/3)\dim V$.
  \item[\rm(b)]  There exists $x \in X$ with $\dim C_V(x) < (1/3) \dim V$.
 \end{enumerate}
\end{cor}

\begin{proof}
Clearly the second statement follows from the first.
To prove the first statement, it suffices to assume that $V$ is irreducible.
So $V=V_1 \otimes \ldots \otimes V_t$ with each $V_i$ an irreducible
$kL_i$-module (and at least one nontrivial). By Lemma~\ref{sl2q even2} we have
that $C_V(X)=0$.  Since $|X|$ is odd, the result follows by 
\cite[Thm. 1.1]{GuMa}.
\end{proof}

Combining the previous results yields:

\begin{cor} \label{simple fixed}
 Let $G$ be a direct product of a finite number of isomorphic nonabelian
 simple groups. Let $k$ be an algebraically closed field and $V$ be 
 a $kG$-module with no trivial composition factors. Then there exists
 $g \in G$ with $\dim C_V(g) \le (1/3) \dim V$.
\end{cor}
 
We can now solve a conjecture stated in the 1966 thesis of Peter Neumann
\cite{Nthesis}. There are two new ingredients in our proof.
The first is the previous result.  However, we also improve his result even 
for solvable groups.  The idea is to consider
the average dimension over some coset of a normal subgroup.
This result does require the irreducibility assumption  (consider 
the augmentation ideal for an elementary abelian $2$-group in
any odd characteristic). 

\begin{thm}   \label{peter}
 Let $G$ be a finite group acting irreducibly and nontrivially on the 
 $kG$-module $V$.There exists an element $g \in G$ with 
 $\dim C_V(g) \le (1/3) \dim V$.
\end{thm}

\begin{proof}
Let $R:=\End_{kG}(V)$. This is a division ring. We may replace
$k$ by the center of $R$ and so $R$ is a central simple algebra over $k$.
Now extend $k$ to a Galois splitting field $L$ of $R$.
Then $V \otimes_k L$ is a direct sum of Galois conjugates of $V$. Any element
has the same size fixed space on each of these conjugates.
We can then extend scalars and assume that $k$ is algebraically closed, so
there is no harm in assuming that $V$ is absolutely irreducible.

Let $N$ be a minimal normal subgroup of $G$.
If $N$ is an elementary abelian $p$-group for $p > 2$, the result
follows by \cite{IKMM} (see also \cite[Thm. 1.1]{GuMa}).

Suppose that $F(G)=1$. Then $N=L_1 \times \ldots \times L_t$ is a direct
product of isomorphic non-abelian simple
groups and $V$ is a completely reducible $kN$-module with no fixed points.
The result follows by Corollary~\ref{simple fixed}.

The remaining case is when $N$ is an elementary abelian $2$-group. If $N$ acts
homogeneously on $V$, then $N$ is central and there is a fixed point free
element in $N$. Let $\Omega =\{V_1, \ldots, V_m\}$ denote the homogeneous
components of $N$ with $m>1$.  Then $G$ acts transitively on $\Omega$. Thus,
there exists $g \in G$ with no fixed points on $\Omega$.

We will prove that $\avg(gN,V) \le (1/3) \dim V$.
This completes the proof for then certainly some element in the coset $gN$
has fixed point space of dimension at most $(1/3) \dim V$.   
Let $\Delta \subseteq \Omega$ be an orbit for $g$ with $\delta = |\Delta|$.
Let $W = \oplus_{i \in \Delta} V_i$. We will in fact prove that
$\avg(gN,W) \le (1/3) \dim W$ and clearly this suffices.
Obviously, $\dim C_W{(gy}) \le (1/\delta) \dim  W$ for every $y \in N$.  
Thus, if $\delta > 2$, our assertion has been proved. It remains to consider
the case that $\delta = 2$. By reordering, we may assume that
$W = V_1 \oplus V_2$.
Consider the image of $N$ in $\GL(W)$. Since the $V_i$ are distinct nontrivial
weight spaces for $N$, the image of $N$ in $\GL(W)$ is $M$, an elementary
abelian group of order $4$. In computing this average on $W$, we may mod out by
the kernel of the action of $N$ on $W$ and so we are averaging over the coset
$gM$ of size $4$. Let $M=\{1, z_1, z_2, z \}$
where $z_i$ is trivial on $V_i$ (and so acts as $-1$ on the other factor).  

Write $g=(s,t)\tau$ in $\GL(W)$ where $\tau$ is an involution interchanging
coordinates (identify $V_1$ with $V_2$).
Then $g^2=(st, ts)$ and so we see that $c:=\dim C_W(g) = \dim C_{V_1}(st)$.   
Note that  $(gz)^2=g^2$ and $(gz_i)^2=g^2z$.
Thus, $\dim C_W(gz)=\dim C_W(g)=c$ and $\dim C_W(gz_i)  \le \dim V_1 - c$  
(since $g^2z$ and $g^2$ have no common fixed points). Thus,
$$\begin{aligned}
  \avg(gN,W) =& \frac{1}{4}(2 \dim C_W(g) + 2 \dim C_W(gz_i))\\
    \le& \frac{1}{4}2\dim V_1 = \frac{1}{4}\dim W < \frac{1}{3}\dim W.
\end{aligned}$$
This completes the proof.
\end{proof}

\begin{rem}   \label{rem:history}\ \hfill
\begin{enumerate}
 \item  Neumann \cite{Nthesis} proved that for solvable $G$ there exists an
  element $g \in G$ with $\dim C_V(g) \le (7/18) \dim V$. Segal and Shalev
  \cite{SS} using the classification of finite simple groups showed that
  for all groups one could obtain $(1/2)\dim V$ as a bound. Isaacs et al
  \cite{IKMM} improved this to $(1/p)\dim V$ as long as $V$ is a completely
  reducible $G$-module (where $p$ is the smallest prime dividing
  the order of $G$  -- for $p=2$, they used the classification of finite simple
  groups via a result of \cite{GK}).   
  Mar\'oti and the first author \cite{GuMa} improve this by showing that one
  can take $V$ to be any module with no trivial composition factors and 
  improve the bound to strictly less than $(1/p)\dim V$ where $p$ is the
  smallest prime divisor of $|G|$.
 \item  As we have already noted, one cannot improve the $1/3$. See the next
  section for examples, showing that even for arbitrarily large dimension,
  one cannot do better than $1/9$.  
\end{enumerate}
\end{rem}

We close this section by extending our result to infinite linear groups as
announced in Theorem~\ref{thm:main2} of the introduction:

\begin{thm} \label{linear groups}
 Let $G$ be a nontrivial irreducible subgroup of $\GL(V)$ with $V$ a finite
 dimensional vector space over a field $k$.  There exists
 $g \in G$ with $\dim C_V(g) \le (1/3) \dim V$.
\end{thm}

\begin{proof}
As in the proof of Theorem~\ref{peter} there is no harm in assuming that  
$V$ is absolutely irreducible.
Let $n:=\dim(V)$. By passing to a subgroup we may assume that $G$ is finitely
generated. By
passing to the subfield of $k$ generated by the matrix entries of the
finitely many generators and their inverses we may assume that $k$ is finitely
generated as a field over the prime field. Indeed, we may assume that
$G\le \GL_n(R)$ where $R$ is a finitely generated subring of $k$ (whose
quotient field is $k$). Let $M$ be a maximal ideal of $R$ such that the
image $\bar{G}$ of $G$ in $\GL_n(R/M)$ is still absolutely irreducible. 
(This can be easily done. Choose $n^2$ elements of $G$ that form a basis
for $M_n(k)$. This is an open condition on the spectrum of $R$ and so holds
for a dense subset of maximal ideals of $R$). \par
By the Nullstellensatz, $R/M$ is finite, so by the result for finite groups
in Theorem~\ref{peter}, we may choose an element $\bar{g} \in \bar{G}$ whose
fixed space has dimension at most $n/3$. This is equivalent to the
rank of $I - \bar{g}$ being at least $(2/3)n$, whence the same is obviously
true for any lift $g \in G$ of $\bar{g}$.
\end{proof}

\section{Some characteristic~0 results} \label{char0}

Under suitable circumstances, one can improve Theorem \ref{thm:main2}.
If $G$ is a compact Lie group (or complex algebraic group), then
as the dimension increases, the dimension of any weight space of 
a maximal torus divided by the dimension tends to $0$ \cite[Thm.~6]{GLM}.

Here, we specialize to irreducible complex representations of non-abelian
simple groups where we prove much better upper bounds for the maximal size of 
eigenspaces of suitable elements.

\begin{thm}   \label{thm:main3}
 For any $\eps>0$ there exists $N>0$ with the following property: for all
 finite quasi-simple groups $G$ there exists $g \in G$ so that for all
 irreducible $\CC G$-modules $V$ of dimension $\dim V > N$ every eigenspace
 of $g$ is of dimension at most $\epsilon \dim V$.
\end{thm}

\begin{proof}  
We first consider the case  that $G$ is quasi-simple of Lie type. Furthermore,
we may assume that $G$ is not one of the finitely many exceptional covering
groups. Now for $G=\SL_2(q)$, $q\ge5$,
let $g$ be a regular semisimple element of odd order $(q-1)/2$ or $(q+1)/2$,
and otherwise let $g$ be (a lift to the central extension $G$ of) a regular
semisimple element of order $\Phi_e^*(q)$ in the maximal torus $T$ of $G/Z(G)$
as given in Table~\ref{tab:classtorus} for classical groups,
respectively as in Table~\ref{tab:exctorus} for exceptional groups. Note that
such regular semisimple elements exist whenever there is a suitable Zsigmondy
prime, that is, in all but finitely many cases. Then $|C_G(g)|=|T|$.
Let $V$ be an irreducible $\CC G$-module, with character $\chi$, and set
$d:=\dim V=\chi(1)$.  Then, the multiplicity of any linear character $\la$ of
$H:=\langle g\rangle$ in $\chi|_H$ is 
$$\langle\chi,\la\rangle_H=\frac{1}{|H|}\sum_{t\in H}\chi(t)\bar\la(t)
  \le\frac{1}{|H|}\Big(d+ |H|\,|T|^{1/2}\Big).
  $$
Thus, any eigenspace of $g$ on $V$ has at most dimension
$(\frac{1}{|H|}+\frac{|T|^{1/2}}{d})\dim V$. Since $|H|=\Phi_e^*(q)$
goes to infinity as $q^e$ does and since $|T|^{1/2}/d \rightarrow 0$ (by
\cite[Table~1]{TZ}), the claim follows.

Next consider $G=\fA_n$.  Let $x$ be a cycle of odd length $n-e$
with $e = 2$ or $3$. For sufficiently large $n$, it follows by
\cite[Thm. 1.2]{LaSh} that $|\chi(x^j)| < \chi(1)^{1/2}$ for all $x^j \ne 1$.
Now let $H$ be the subgroup generated by such an $x$.
Arguing as we did for the case of Lie type groups, we see that this implies
that  the multiplicity of any  irreducible character  of $H$ in the restriction
of a nontrivial character $\chi$ of $\fA_n$ is less than
$\chi(1)/(n-e) + \chi(1)^{1/2}$. The result follows. 

Suppose that $G$ is the double cover of $\fA_n$.    Let $\pi$ be
the projection from $G$ onto $\fA_n$.  Let $x \in G$ with $\pi(x)$
a product of a $2$-cycle and $2^f$-cycle where $n/2 \le 2^f \le n-4$. 
It follows by \cite[Thm. 3.9]{HH} that every noncentral element $y$ of
$H=\langle x \rangle$ is conjugate to $yz$ where $z$ is the central involution
in $G$. Thus $\chi(y)=0$ for every such $y$ if $\chi$ is a faithful character
of $G$.  
This implies that every eigenspace of $x$ is either $0$ dimensional
or has dimension $\chi(1)/2^f$.

Since the result is asymptotic, we need not consider sporadic groups or the
covers of $\fA_6$ and $\fA_7$.
\end{proof}

We give two sample results showing how the bound obtained in the proof of the
previous theorem can be strengthened with a bit more work for the groups of
Lie type. 
 
\begin{prop}
 Let $G=E_8(q)$. Then there exists $g\in G$ such that for every irreducible
 $\CC G$-module $V$ the dimension of every eigenspace of $g$ on $V$ is at
 most $(1/q^8)\dim V$.
\end{prop}

\begin{proof}
Let $g\in G$ of order $\Phi_{30}(q)=q^8+q^7-q^5-q^4-q^3+q+1$. Then $g$
generates a TI-torus $T$ of $G$, with $|N_G(T)/T|=30$ (see
Table~\ref{tab:exctorus}). Let $\chi\in\Irr(G)$ be
a non-trivial character. By Proposition~\ref{prop:values} we have
$|\chi(s)|\le 30$ for all powers $s=g^b\ne1$ of $g$. Thus, the multiplicity
of any linear character of $T$ in $\chi|_T$ is at most
$$\frac{\chi(1)+30|T|}{|T|}=\frac{\chi(1)}{|T|}+30\le \frac{\chi(1)}{q^8}$$
as claimed, since $\chi(1)>q^{28}$ by \cite{Lue}.
\end{proof}

\begin{prop}
 Let $G=\OO_{2n}^-(q)$. Then there exists $g\in G$ such that for every
 irreducible $\CC G$-module $V$ the dimension of every eigenspace of $g$ on
 $V$ is at most $(1/\Phi_{2n}^*(q))\dim V+n$.
\end{prop}

\begin{proof}
Consider $G$ as the derived subgroup of ${\PCO_{2n}^\circ}^-(q)$ and let
$g\in G$ of order $\Phi_{2n}^*(q)$. Then $g$ is a regular element in a cyclic
maximal torus of ${\PCO_{2n}^\circ}^-(q)$ of order $q^n+1$, and any non-trivial
power of $g$ is conjugate to precisely $n$ of its powers. Let $\chi\in\Irr(G)$
be a non-trivial character. By Proposition~\ref{prop:values} we have
$|\chi(s)|\le n$ for all powers $s=g^b\ne1$ of $g$. Thus, the multiplicity
of any linear character of $H:=\langle g\rangle$ in $\chi|_H$ is at most
$$\frac{\chi(1)+n|H|}{|H|}=\frac{\chi(1)}{|H|}+n$$
as claimed.
\end{proof}

The following examples show that (as opposed to the simple group case)
even in characteristic $0$, one cannot get arbitrarily small ratios for
$\dim C_V(g)/\dim V$ as $\dim V$ increases (and $V$ is irreducible).

\begin{exmp} \label{exmp:A5}
Let $L=\fA_5$ and let $W$ be an irreducible $5$-dimensional module in
characteristic $0$  (all we need assume is that the characteristic is not $2$).

Set $G(m) =L \times \ldots \times L$ ($m$ copies) acting on
$V(m) =W \otimes \ldots \otimes W$ ($m$ copies).
Then $\dim C_{V(m)}(g) > (1/50) \dim V(m)$ for all $g \in G(m)$.

\begin{proof}  
If $g$ has order $5$, then 
$\chi_{V(m)}(g)=0$ and  $\dim C_{V(m)}(g)  = (1/5)\dim V(m)$ for all $m$.
 
If $g$ has order $2$, we claim that $\dim C_{V(m)}(g) > (1/2) \dim V(m)$. By
induction on $m$, it suffices to assume that $g=(h, \ldots, h)$ with $h\in L$
of order $2$. Then $\chi_{V(m)}(h) = 1$, whence the trivial eigenspace of $g$
has dimension  $(1/2)(\dim V(m) + 1)> (1/2) \dim V(m)$ (but as $m$ increases,
the ratio gets arbitrarily close to~$1/2$).

If $g$ has order $3$ , then $\chi_{V(m)}(h) = \pm 1$, whence the trivial
eigenspace of $g$ has dimension at least
$(1/3)(\dim V(m) - 2) \ge (1/5) \dim V(m)$ (and $1/5$ is achieved for $m=1$). 

Now take $g$ arbitrary. We may write (up to reordering)
$g=g_1 \otimes g_2 \otimes g_3 \otimes g_5$ acting on
$V(a_1)\otimes V(a_2)\otimes V(a_3) \otimes V(a_5)$  with $m=\sum a_i$ and
$g_i$ has each component of order $i$. Thus, $g_1$ is trivial on $V(a_1)$ and
by the previous results, $\dim C_{V(a_i)}(g_i) \ge (1/5) \dim V(a_i)$ for
$i=3,5$ and $\dim C_{V(a_2)}(g_2)>(1/2)\dim V(a_2)$. Thus, the result follows.
\end{proof}
\end{exmp}

We get a similar result for $\fA_4$. The proof is essentially the same
as in the previous example:

\begin{exmp}   \label{exmp:A4}
 Let $L=\fA_4$ and let $W$ be the irreducible $3$-dimensional module in
 characteristic not $2$. 
 Set $G(m) =L \times \ldots \times L$ ($m$ copies) acting on
 $V(m) =W \otimes \ldots \otimes W$ ($m$ copies).
 Then $\dim C_{V(m)}(g) \ge (1/9) \dim V(m)$ for all $g \in G(m)$.
\end{exmp}

We next show that for direct products of simple groups and for sufficiently
large dimension, Example~\ref{exmp:A5} is the worst case.

\begin{lem}
 Let $L$ be a finite nonabelian simple group and $V$ an irreducible nontrivial
 $\CC L$-module. Let $G = G(m)$ be the direct product of $m$ copies of $L$
 and $W=V(m)$ be the tensor product of $m$ copies of $V$. Let $\epsilon > 0$.
 If $m$ is sufficiently large, then there exists an element $g \in G$ such
 that $\dim C_{V(m)}(g) < (1/50 + \epsilon) \dim V(m)$.
\end{lem}

\begin{proof}
Let $\chi$ be the character of $V$. Let $e$ be the exponent of $L$ 
(i.e. the smallest positive integer $e$ such that $x^e=1$ for all $x \in L$).
Choose $g_1, \ldots, g_s \in L$ such that the least common multiple of the
orders of the $g_i$ is $e$. Set $y = (g_1, \ldots, g_s)$ in $G(s)$. Consider
$(y,\ldots, y) \in G(st)$ acting on $V(st)$. We see that for any $0<j<e$ we
have $|\chi(y^j)|/\chi(1)^{st} \rightarrow 0$ as $t$ increases. It follows
that $V(st)$ is very close to a free module for $\langle y \rangle$ and in
particular 
$$
  \lim_{t\rightarrow\infty}\frac{\dim C_{V(st)}(y)}{\dim V(st)} = \frac{1}{e}.
$$
Thus, if $m$ is sufficiently large, we can find $g \in G$ satisfying the
conclusion unless $e \le 50$. An easy inspection shows that $e>50$ is true for
every $L$ aside from $L = \fA_5$.   

So now suppose that $L=\fA_5$ (note that $e=30$). So $\dim V = 3, 4$ or $5$.
If $\dim V=4$, an element of order $3$ has no fixed points and so the same is
true for an element of order $3$ of the form $(x,1, \ldots, 1) \in L(m)$
for any $m$. If $\dim V=5$, then $V$ is a free module for an element of order
$5$.  An element of order $3$ has  a $1$-dimensional fixed space. If $m$ is
large enough, we can choose an involution whose fixed space has dimension
very close to $(1/2) \dim V$.  Arguing as above, we can choose an element
$g \in L(m)$ whose fixed point space has dimension as close as we want to
$(1/50) \dim V(m)$, whence the result holds in this case.

Finally suppose that $\dim V = 3$.  Every element of order $3$ in $L(t)$ has
a fixed space of dimension $3^{t-1}$. Consider $g=(h,h^2)$ with $h$ of
order~$5$ acting on $V(2)$. Then the fixed space of $g$ is $1$ dimensional
and so $(1/9) \dim V(2)$. An involution has a fixed space of dimension $1$
on $V$.  Thus, we see that there is an element of order $30$ acting on $V(m)$
for any $m \ge 4$ with $\dim C_V(g) = (1/81) \dim V(m)$.
\end{proof}

We can extend the previous result to the case of primitive groups.

\begin{thm}
 Let $G$ be a finite group and $W$ an irreducible primitive $\CC G$-module.  
 Let $s$ be any positive number greater than $1/50$. If $\dim W$ is
 sufficiently large, then there exists $g \in G$ such that
 $\dim C_W (g) < s \dim W$.
\end{thm}

\begin{proof}
We may assume that $W$ is faithful. Any normal subgroup of $G$ acts
homogeneously on $W$, whence any abelian normal subgroup of $G$ is central.
If $Z(G) \ne 1$, the result is clear.  So we may assume that $F(G)=1$.
So $F^*(G) = L_1 \times \ldots \times L_s$ with $L_i$ simple and every
irreducible $F^*(G)$-submodule of $W$ is isomorphic to
$V_1 \otimes \ldots \otimes  V_s$ where $V_i$ is a nontrivial irreducible
$L_i$-module. If $\dim V_i$ is large enough, it follows by
Theorem~\ref{thm:main3} that there exists $g_i \in L_i$ with fixed space of
dimension less than $(1/50) \dim V_i$, a contradiction. Thus there are only
finitely many possible isomorphism types for the $L_i$.   
We may assume that $s$ is as large as we wish (otherwise $F^*(G)$ has bounded
order, whence so does $|G|$ and so $\dim V$ is bounded). So we may assume
that $L_i\cong  L_1$ for $1\le i\le m$ for $m$ as large as we wish and that
$V_i \cong V_1$ for those $i$ (since there are only a bounded number of
possible isomorphism types of irreducible modules). By the previous lemma,
for $m$ sufficiently large, there exists $(g_1,\ldots,g_m)\in L_1(m)$ with
$\dim C_{V(m)}(g) < s \dim V(m)$ with $V=V_1$.
Setting $g= (g_1, \ldots, g_m,1 \ldots, 1) \in F^*(G)$, we see that
$\dim C_W(g) < s \dim W$, as required.
\end{proof}

\section{Products of classes and powers}   \label{sec:cover}

A conjecture of Thompson asserts that for any finite non-abelian simple group
there is a conjugacy class $C$ such that $G=C\cdot C$, that is, any element of
$G$ is the product of two elements in $C$. A proof of this seems out of reach
at present, although we expect that, in some sense, most classes will do.
We propose the following partial results. First we consider rank $1$ groups.
Here, we set $G^\#:=G\setminus\{1\}$.

\begin{thm}   \label{thm:2B2}
 Let $G$ be a rank $1$ finite simple group of Lie type. Let $C$ be the
 $G$-conjugacy class of an element $x$ of order $o(x)>2$. Assume that one of
 \begin{enumerate}[\rm(1)]
  \item $G=\PSL_2(q)$ with $q$ odd and $x$ is not unipotent;
  \item $G=\PSL_2(q)$ with $q$ even and $o(x)$ does not divide~$q+1$;
  \item $G=\tw2G_2(q^2)$, $q^2\ge 27$, and $o(x)$ is not divisible by~3;
  \item $G=\tw2B_2(q^2)$, $q^2\ge8$; or
  \item $G=\PSU_3(q)$ and $x$ is regular of order $(q^2-q+1)/d$ or
   $(q^2-1)/d$, with $d=\gcd(3,q+1)$.
 \end{enumerate}
 Then $G^\# \subseteq CC$.
\end{thm}

\begin{proof}
If $G=\PSL_2(q)$, this is a straightforward matrix computation, see also
Macbeath \cite{Mac}.
For the Suzuki groups $\tw2B_2(q^2)$ the generic character table is available
in \Chevie\ and the claim can be checked by computer. (By \cite[Thm.~2]{Gow} we
actually only need to worry about the classes of elements of order~4.)
Similarly, for the Ree groups $\tw2G_2(q^2)$, $q^2\ge27$, by loc{}. cit{}. and
our restriction on $C$ we only have to check that the square of any class of
semisimple elements of order bigger than~2 meets all non-semisimple classes,
which is immediate using the table in \Chevie.  \par
Finally, the character table of $\PSU_3(q)$ can also be found in \Chevie, and
we only need to verify that squares of regular semisimple classes of order
$(q^2-q+1)/d$ or $(q+1)/d$ contain all non-semisimple classes.
\end{proof}

This has the following immediate consequence:

\begin{thm}   \label{rank1}
 Let $w_1$ and $w_2$ be nontrivial words in the free group on $d$ generators.
  \begin{enumerate}
  \item[\rm(a)] If $G=\tw2B_2(q^2)$, $q^2\ge 8$ and $w_i^2(G^d) \ne 1$ for
   $i=1,2$, then $G=w_1(G^d)w_2(G^d)$.
  \item[\rm(b)] If $G=\tw2G_2(q^2)$, $q^2\ge 27$, and
   $w_i^9(G^d)\ne 1\ne w_i^6(G^d)$ for $i=1,2$, then $G=w_1(G^d)w_2(G^d)$.
 \end{enumerate}
\end{thm}

The main result of \cite{LST}  is a version of this for any pair of words
$w_1$ and $w_2$ and any sufficiently large non-abelian simple group (where
sufficiently large depends on the choice of $w_1$ and $w_2$).

Next we consider some rank $2$-groups. 

\begin{thm}   \label{thm:3D4}
 Let $G$ be a simple group $\PSL_3(q)$ ($q>2$), $\PSp_4(q)$, $G_2(q)$,
 $\tw3D_4(q)$ or $\tw2F_4(q^2)'$. Then there exists a conjugacy class $C$ of
 semisimple elements of $G$ such that $G^\#\subseteq C C$.
\end{thm}

\begin{table}[htbp] 
  \caption{Conjugacy classes in some rank-2 groups}
  \label{tab:excclass1}
\[\begin{array}{|l|c|}
\hline
  G& o(x)\\ \hline\hline
 \PSL_3(q)& (q^2+q+1)/(3,q-1)\\ 
          & \text{or }(q^2-1)/(3,q-1)\\ \hline
 \PSp_4(q),\ q\text{ even}& q^2-1\\ \hline
 \PSp_4(q),\ q\text{ odd}& (q^2+1)/2\\ \hline
 G_2(q),\ q\ge3& (q^2+q+1)/(3,q-1)\\ \hline
 \tw3D_4(q)& q^4-q^2+1\\ \hline
 \tw2F_4(q^2)& \Phi_{24}'\\ \hline
\end{array}\]
\end{table}

\begin{proof}
Let $C$ be the conjugacy class of a regular semisimple element $x$ of order
as given in Table~\ref{tab:excclass1}. We claim that the $(C,C,C')$ structure
constant is non-zero for any conjugacy class $C'$ of $G$. For the first four
families of groups, the complete generic character table is available in the
\Chevie-system, and the claim can be checked by computer. \par
The group $\tw2F_4(2)'$ is covered by the square of class 13a, so it remains to
consider $\tw2F_4(q^2)$ for $q^2\ge8$. Here, the complete character table is
known in principle, but is only partly contained in \Chevie. Now note
that for any class $C'$ of semisimple elements, our claim follows by the
elementary observation \cite[Thm.~2]{Gow}. By Lemma~\ref{lem:Lvalues} the
irreducible characters not vanishing on $C$ are the irreducible
Deligne-Lusztig characters $R_{T,\theta}$, with $|T|=\Phi_{24}'$, and some
of the unipotent characters.
\par
Let $y$ be a non-semisimple element of $G$. By \cite[Prop.~7.5.3]{Ca} we
have $R_{T,\theta}(y)=0$ unless the semisimple part of $y$ is conjugate to
an element of $T$. Since all non-identity elements of $T$ are regular, this
implies that $R_{T,\theta}(y)=0$ unless $y$ is unipotent. In the latter case,
$R_{T,\theta}(y)$ does not depend on $\theta$ by \cite[Cor.~7.2.9]{Ca}.
Let $I:=\{\pm R_{T,\theta}\mid1\ne\theta\in\Irr(T)\}\subset\Irr(G)$ and
$I':=\Irr(G)\setminus I$. By the orthogonality relations we have
$|\sum_{\chi\in I}\chi(x)^2|<|T|$. Writing $a:=R_{T,\theta}(y)$,
$d:=R_{T,\theta}(1)$ we find
$$\begin{aligned}
  |n(C,C,C')|
  &=\frac{|C|^2}{|G|}\left|\sum_{\chi\in I}\frac{\chi(x)^2\chi(y)}{\chi(1)}
   +\sum_{\chi\in I'}\frac{\chi(x)^2\chi(y)}{\chi(1)}\right|\cr
  &\ge\frac{|C|^2}{|G|}\left(\Big|\sum_{\chi\in I'}
   \frac{\chi(x)^2\chi(y)}{\chi(1)}\Big| -\frac{|a|}{d}|T|\right).\\
\end{aligned}$$
Now $a=R_{T,\theta}(y)$ is the value of a Green function and these have been
determined by the second author \cite{Ma90}, from which we get
$|a|\le q^{16}+2q^{15}$. Since
$|T|\le q^4+2q^3$ and $d=|G|_{2'}/|T|\ge q^{20}$ we have $|a|\cdot|T|/d<1/2$.
It can now be computed from the unipotent part of the character table in
\Chevie\ that $n(C,C,C')>0$ for all unipotent classes $C'$.
\end{proof}

\begin{rem}
It's easily seen that for all simple groups in Theorems~\ref{thm:2B2}
and~\ref{thm:3D4} at least one of the classes is real, so Thompson's
conjecture holds for all these groups.
\end{rem}

For the other types of simple groups, we give an approximation by showing that 
$G^\#$ is covered by a product of two classes.
\par
For most families of simple classical groups of Lie type this was already
shown in \cite[Thm.~2.1--2.6]{MSW}:

\begin{thm}[Malle--Saxl--Weigel (1994)]   \label{thm:MSW}
 Let $G$ be a simple classical group of Lie type not of type $\OO_{4n}^+(q)$
 and different from $\PSU_3(3)$ and $\OO_8^-(2)$. Then there are two conjugacy
 classes of $G$ whose product covers $G^\#$.
\end{thm}

It is straightforward to check from their character tables that this continues
to hold for $\PSU_3(3)$ and $\OO_8^-(2)$.

We next handle the remaining family of classical groups -- giving a small
improvement of the argument in \cite[\S7]{LST} for $\OO_{4n}(q), n > 2$.

\begin{thm}   \label{thm:O4n}
 Let $G$ be a simple classical group of Lie type $\OO_{4n}^+(q)$ with
 $n \ge 2$. Then there exist two conjugacy
 classes of $G$ whose product covers $G^\#$.
 \end{thm}
 
\begin{proof}
If $n=2$, we choose the two classes as in \cite[Thm.~2.7]{MSW}, so that only
three irreducible characters do not vanish on either class, which moreover
are unipotent. The values of the unipotent characters of a group of type
$D_4$ are contained in \Chevie\ and one computes that the result holds.
 
Now assume that $n \ge 3$.   
Let $C_1$ be the class of an element of
order a Zsigmondy prime divisor of $q^{2n-1}-1$ inside the subgroup
$\GL_{2n-1}(q)$ and for $C_2$ the class of order the product of a
Zsigmondy prime divisor of $q^{2n-2}+1$ with a Zsigmondy prime divisor of
$q^2+1$, inside a subgroup $\Omega_{4n-4}^-(q)\times\Omega_4^-(q)$, as in
\cite[\S7]{LST}. The claim is proved in the last section of that paper aside
from a small number of cases. There are precisely $3$ nontrivial characters
that vanish on neither $C_1$ nor $C_2$ described there (the Steinberg
character, $\gamma$ and $\delta$) which take on the value $\pm 1$ on each
$C_i$. We want to show that any nontrivial class $C_3$ is contained in
$C_1C_2$. By a result of Gow \cite{Gow}, we may assume that $C_3$ consists of
non-semisimple elements and so the Steinberg character vanishes on $C_3$.
By a general bound of Gluck \cite[Thm.~1.11 and~5.3]{Gl},
$|\gamma(C_3)| \le (19/20)\gamma(1)$. We use the trivial estimate that 
$|\delta(C_3)| \le |C_G(z)|^{1/2} < (1/20)\delta(1)$ for $z \in C_3$.     
\end{proof} 

Next, we extend this result to the exceptional groups of Lie type. (For those
of small rank, it has already been shown above).

\begin{thm}   \label{thm:E7}
 Let $G$ be a simple group $F_4(q)$, $E_6(q)$, $\tw2E_6(q)$, $E_7(q)$ or
 $E_8(q)$. There exist two conjugacy classes of semisimple elements of $G$
 whose product covers $G^\#$.
\end{thm}

\begin{table}[htbp] 
  \caption{Conjugacy classes in exceptional groups}
  \label{tab:excclass2}
\[\begin{array}{|l|cc|l|}    
\hline
 G& o(x_1)& o(x_2)& \chi\\
\skipa \hline \hline
       F_4(q)& \Phi_{12}& \Phi_8& 1,F_4[i],F_4[-i],\St\\ \hline
       E_6(q)& \Phi_9/(3,q-1)& \Phi_8& 1,\St\\ \hline
   \tw2E_6(q)& \Phi_{18}/(3,q+1)& \Phi_8& 1,\St\\ \hline
       E_7(q)& (\Phi_2\Phi_{18})_{\{2,3\}'}& \Phi_7& 1,\St\\ \hline
       E_8(q)& \Phi_{20}& \Phi_{24}& 1,E_8[i],E_8[-i],\St\\ \hline
\end{array}\]
\end{table}

\begin{proof}
For each type, we choose $C_i$ to contain semisimple elements $x_i$
of order as indicated in Table~\ref{tab:excclass2}. (Such elements exist since
$G$ contains maximal tori with cyclic subgroups of that order.) \par
We first determine the irreducible characters of $G$ simultaneously not
vanishing on both classes. It is easy to check that no non-trivial element of
$G_\ad$ has centralizer order divisible by Zsigmondy primes $p_1,p_2$ for both
element orders (for example, for $G=F_4(q)$ using that the only semisimple
algebraic group of rank at most four whose order polynomial is divisible by
both $\Phi_8$ and $\Phi_{12}$ is $F_4$). Thus, by Lemma~\ref{lem:Lvalues} only
unipotent characters may possibly take non-zero values on both $C_1,C_2$.
\par
The list in \cite[13.9]{Ca} shows that apart from the two, respectively four
characters $\chi$ listed in Table~\ref{tab:excclass2} (where $1_G,\St$ denote
the trivial and the Steinberg character, and otherwise the notation is as in
loc.~cit.) all other unipotent characters have degree divisible by at least
one of the two Zsigmondy primes $p_i$, hence are of $p_i$-defect~0. Thus,
only the listed characters may potentially not vanish on both $C_1,C_2$.
\par
Now let $C$ be any non-trivial conjugacy class of $G$. We use the
character formula (\cite[Thm.~I.5.8]{MM}) to estimate the structure
constant $n(C_1,C_2,C)$. The value of the Steinberg character $\St$ on
semisimple elements equals the $p$-part of their centralizer order, up to
sign, and is zero on all other elements (see \cite[Thm.~6.4.7]{Ca}). The
elements in both classes $C_1,C_2$ are regular, so $|\St(x_1)|=1$. Also,
the values of the characters $F_4[\pm i]$ and $E_8[\pm i]$ on our regular
semisimple elements have absolute value~1 (for example by block theory for
the Zsigmondy primes $p_i$). Thus we are done if we can show that
$|\chi(C)|<\chi(1)/3$ for the one or three non-trivial unipotent characters.
\par
Now $\St(1)=q^{24},q^{36},q^{36},q^{63},q^{120}$ respectively, and
$F_4[\pm i](1)>q^{20}$, $E_8[\pm i](1)>q^{104}$. On the other hand,
$|\chi(x)|\le \sqrt{|C_G(x)|}$ for $x\in C$. The largest centralizers of
non-trivial elements in $G$ have order at most
$q^{36},q^{56},q^{56},q^{99},q^{190}$ respectively, so we are done.
\end{proof}

Finally, we note that:

\begin{prop}   \label{prop:altspor}
 Let $G$ be a simple sporadic group or $\fA_n$, $n\ge 7$. Then there exist two
 conjugacy classes $C$ of $G$ whose elements have odd coprime order such that
 $G^\#\subseteq CC$.
\end{prop}

\begin{proof}
If $G$ is sporadic, this is straightforward using the character tables. In
fact, we can always choose both classes to contain elements of prime order
$p>2$ with this property.  \par
Similarly, this is clear for $\fA_n$, $7\le n\le 16$, with elements of orders
$$(5,7),(3,7),(3,7),(5,7),(5,11),(5,11),(11,13),(11,13),(11,13),(7,13)$$
respectively. If $n \ge 17$, we can take two classes of $\ell$-cycles with
$3n/4 \le \ell \le n-2$ by Bertram \cite[Thm.~1]{Be}, and two such odd
coprime $\ell$ exist unless $n=18$. In the latter case we can also take $C$ to
contain elements of order~17. A random computer search shows that the square
of one of these already covers $\fA_{18}^\#$.
\end{proof}

\begin{proof}[Proof of Theorem \ref{thm:main5}]
For all finite non-abelian simple groups $G$ other than $\PSL_2(q)$, $q=7$ or
$17$, we produce two elements whose
order is prime to~6 such that their classes $C_1,C_2$ satisfy
$G^\#\subseteq C_1C_2$. 
If $G$ is sporadic or an alternating group $\fA_n$ with $n\ge7$, the
claim is an immediate consequence of Proposition~\ref{prop:altspor}. For
$n=5$ we use that the product of the classes $5a$ and $5b$ covers
$\fA_5^\#$, and similarly for $\fA_6$.

So assume that $G$ is a simple group of Lie type. If $G$ has rank $1$, the
result follows by Theorem~\ref{thm:2B2}. Indeed, for $\PSL_2(q)$ with $q$ odd,
we take for $x$ an element of order prime to~6 in one of the two maximal
tori. Such elements will exist unless $q\in\{5,7,17\}$. The latter two cases
have been excluded in the statement of~(b), and the group
$\PSL_2(5)\cong\fA_5$ was already settled.
For $\PSL_2(q)$ with $q\ge8$ even, take for $x$ an element of order
dividing $q-1$ and prime to~3. For $\tw2G_2(q^2)$, take any elements of order
prime to~6, for $\tw2B_2(q^2)$ take any elements of odd order; finally, for
$\PSU_3(q)$, $q>2$, take elements of order $(q^2-q+1)/\gcd(3,q+1)$ (which is
prime to~6). \par
For the groups of rank~2 in Theorem~\ref{thm:3D4}, we may take for $C$ the
class occurring in Table~\ref{tab:excclass1}, except for $\PSp_4(2^f)$
which is covered by the product of any two distinct classes of elements of
order $q^2+1$. The remaining exceptional groups
are covered by the product of two classes as in Table~\ref{tab:excclass2},
which consist of elements of order prime to~6. \par
For the groups occurring in Theorem~\ref{thm:MSW} which have not yet been
discussed, it is straightforward to check from \cite{MSW} that the two
conjugacy classes can be taken to have order a Zsigmondy prime bigger than~3
except possibly for
$$\PSL_6(2),\ \PSL_7(2),\ \PSU_4(2),\ \PSU_6(2),\ \PSU_7(2),\ \PSp_6(2),\
  \PSp_{12}(2).$$
Now $\PSL_6(2)^\#$ is covered by the square of class $31a$, $\PSL_7(2)^\#$ by
the square of $127a$, $\PSU_4(2)$ by the square of $5a$,
$\PSU_6(2)$ and $\PSp_6(2)$ by the square of $7a$, and $\PSp_{12}(2)$ by the
square of class $31a$. For $G=\PSU_7(2)$ the character table is not available
in GAP. We claim that the product of classes of elements of orders~43 and~11
covers $G^\#$. First, by Lemma~\ref{lem:Lvalues} it is easy to see that
only (some) unipotent characters take non-zero values on both classes. These
are contained in \Chevie\ and the claim can then be verified.
The group $\OO_8^-(2)$ excluded in Theorem~\ref{thm:MSW} is covered by the
square of class $17a$.
\par
For the 8-dimensional orthogonal groups of plus type, the classes in
\cite[Thm.~2.7]{MSW} can be chosen to contain elements of order prime to~6,
except for $\OO_8^+(2)$. The latter group is covered by the square of class
$7a$. For $n\ge3$, the two classes for $\OO_{4n}^+(q)$ chosen in
Theorem~\ref{thm:O4n} consist of elements of order prime to~6. \par
Finally, it is straightforward to see that any element of $\PSL_2(q)$ is a
product of two unipotent elements.
\end{proof}

In \cite{LST}, a weaker version of Theorem \ref{thm:main5} was proved.  

We obtain the following consequence:

\begin{thm}
 Let $G$ be a finite non-abelian simple group. Let $m$ be a prime power.
 Then every element of $G$ is a product of two $m$th powers.
\end{thm}

\begin{proof}
By Theorem~\ref{thm:main5} we may assume that $\gcd(m,6)=1$.
Now for $G$ sporadic or $G=\fA_n$, $n\ge7$, our claim follows from
Proposition~\ref{prop:altspor}. The groups $\fA_5$ and $\fA_6$ are both
covered by products of two 3-elements or two 5-cycles.

So assume that $G$ is a group of Lie type. Let $B$ denote a Borel subgroup of
$G$ and $U^-$ the opposite of the unipotent radical of $B$.
If $\gcd(m,|B|)=1$, we use Chernousov--Ellers--Gordeev \cite[Thm.~2.1]{CEG}
which asserts that every element of $G$ is conjugate to an element of $U^-B$,
and so a product of two $m$th powers.

So now assume that $\gcd(m,|B|)\ne1$. For all groups $G$ of rank~1 except
possibly $\PSL_2(2^f)$ by Theorem~\ref{thm:2B2} there exists a conjugacy
class of elements of order prime to $|B|$ whose square covers $G^\#$. For
$\PSL_2(2^f)$ it's easily seen that any element is a product of two elements
of order dividing $q+1$. Similarly, for the groups in Theorem~\ref{thm:3D4}
the claim is clear except for $G=\PSp_4(2^f)$. Here, again $G^\#$ is also
covered by products of two elements of order dividing $q^2+1$. For all other
groups, we have already argued in the previous proof that they can be covered
by products of two elements whose order is prime to $|B|$.
\end{proof}

\section{Further generation results}

We use our results on overgroups of suitable cyclic subgroups in simple
groups to derive the following:

\begin{thm}   \label{thm:generation}
 Let $S$ be a finite simple group different from $\OO_8^+(2)$. There exists
 a conjugacy class $C$ of $S$ consisting of elements of order prime to~6 such
 that if $1\ne x\in S$, then $S=\langle x,y\rangle$ for some $y\in C$. 
\end{thm}

In the one exception, we can instead take $C$ to be a class of elements of
order $15$ (and so in all cases, there is a class $C$ of elements of odd order
satisfying the conclusion).  Moreover,  if $1 \ne x \in S$ has odd order,
then we can choose the class to consist of elements of order $7$.

Before we begin the proof, we point out some easy consequences
including the affirmative answer to a question of Getz \cite{getz}.   
Recall that a quasi-simple group is a perfect central cover of a simple group.

\begin{cor} \label{getz}
 Let $S$ be a finite quasi-simple group. If $T$ is a solvable subgroup of $S$,
 then there exists a solvable subgroup $T_1 \ge T$ and an element $s \in S$ of
 order prime to $6$ such that $S = \langle T_1, s \rangle$.
\end{cor}

\begin{proof}
It is trivial to reduce to the case $S$ is simple  (elements of order
prime to $6$ in the simple group can be lifted to elements of order prime
to $6$ in a covering group).

The result is now immediate unless $S = \OO_8^+(2)$.
Consider that case. Let $T_1$ be a maximal solvable subgroup of $S$
containing $T$. The result follows by the remarks above unless $T_1$
is a $2$-group.  However, it is clear that a Sylow $2$-subgroup of $S$  is not
a maximal solvable subgroup of $S$ (it is contained in a minimal parabolic 
subgroup which is still solvable).
\end{proof}

\begin{cor}
 Let $S$ be a finite quasi-simple group.  Then $S$ can be generated
 by two conjugate elements of order prime to $6$.
\end{cor}  

In the proof of the theorem, we need the following observation (see
Guralnick \cite[2.2]{Gu97} for a slightly different proof):

\begin{lem}   \label{lem:2max}
 Let $C$ be a non-trivial conjugacy class in a finite simple group $G$.
 Then $C$ is not contained in the union of any two proper subgroups.
\end{lem}

\begin{proof}
Suppose that $C$ is contained in $X\cup Y$, with $X,Y$ proper subgroups.
Replacing $X$ and $Y$ by maximal subgroups, we still have the hypothesis and
$C$ is not contained in either $X$ or $Y$ since clearly $G=\langle C\rangle$. 
\par
Choose $x\in C \cap (X\setminus Y)$. Note that if $y\in Y$, then $x^y\in X$
(for $x^y\in Y$ implies that $x\in Y$). Thus $\langle x^Y\rangle$ is contained
in $X$ and normalized by $Y$ obviously and also by $x$. Since $Y$ is maximal,
$G=\langle Y,x\rangle$ normalizes the proper subgroup
$\langle x^Y\rangle\le X$, a contradiction.
\end{proof}

\begin{exmp}
Note that one cannot replace 2 by 3 in the lemma: take $G = \SL_n(2)$ or
$\Sp_{2n}(2)$ and
let $X,Y,Z$ be the stabilizers in $G$ of three vectors in a 2-space. Then every
transvection of $G$ fixes a hyperplane and so fixes at least one of the
vectors in this 2-space.
Thus the class of transvections is contained in the union of three subgroups.
A similar argument applies to the class of transpositions in the non-simple
group $\fS_n$.
\end{exmp}

\begin{proof}[Proof of Theorem~\ref{thm:generation}]
We use two types of standard arguments. In the first situation, let's call it
(S1), we let $y$ be an element as in Theorem~\ref{thm:main}(a), of order prime
to~6. If there are no more than two maximal subgroups of $S$ containing $y$,
then by Lemma~\ref{lem:2max} any conjugacy class in $S\setminus\{1\}$ contains
an element $x$ outside of these maximal subgroups, whence
$S=\langle x,y\rangle$. In some of the cases where there are more than
$2$ maximal subgroups, we compute (as in \cite{BGK}) directly
that no conjugacy class is contained in the union of the maximal subgroups
containing $y$. 
\par
In the second situation (S2), we produce two conjugacy
classes $C_2,C_3$ of $S$, at least one of them containing elements of order
prime to~6, such that the $(C,C_2,C_3)$-structure constant is non-zero for
any non-trivial conjugacy class $C$ of $S$ and such that no maximal subgroup
of $S$ can contain elements from both $C_2$ and $C_3$, in which case we're
done again.  
\par
Assume first that $S$ is sporadic. Let $y$ be an element of prime order as
in Table~\ref{tab:sporadic}. Then unless
$$S\in\{M_{12},\ Ti,\ HS,\ McL,\ Suz,\ Fi_{22}\}$$
we are in situation~(S1). It is easily checked from the character tables that
in each of the six exceptions $S$ listed above, for any prime $p$ there are
less elements of order~$p$ in the (disjoint) union of the relevant maximal
subgroups than in any class of elements of order~$p$ in $S$, except for
elements of order~5 in $HS$, elements of order~3 in $Suz$, and elements of
order~2 or 3 in $Fi_{22}$. In $HS$, we reach situation~(S2) with $C_2=11a$
and $C_3=20a$; and in both $Suz$ and $Fi_{22}$ with $C_2=11a$ and $C_3=13a$.
\par
Now let $S$ be of exceptional Lie type. Here, we take $y$ an element generating
a cyclic subgroup as in Table~\ref{tab:exctorus}. Note that the order of $y$
is coprime to~6 unless possibly when $S=E_7(q)$. Again, $y$ is contained in
at most two maximal subgroups, except for $S\in\{G_2(3),\ F_4(2)\}$,
and we may conclude as before. For $G_2(3)$ and $F_4(2)$ we are in situation
(S2) with the classes $9a,13a$, respectively $13a,17a$.
For $S=E_7(q)$, let $C$ denote a non-trivial conjugacy class. We
have shown in Proposition~\ref{thm:E7} that the $(C,C_2,C_3)$-structure
constant is non-zero, where $C_2$ contains elements of order $\Phi_7$ and
$C_3$ contains elements as in Table~\ref{tab:exctorus}, whence there exist
$(x,y,z)\in(C,C_2,C_3)$ with $xyz=1$. Since the order of the maximal subgroup
$\tw2E_6(q).D_{q+1}$ is not divisible by (a Zsigmondy prime divisor of)
$\Phi_7$, we necessarily have $S=\langle y,z\rangle=\langle x,y\rangle$, so
we are in situation (S2).
Clearly, $y$ has order prime to~6.
\par
Now consider  $S=\fA_n$, $n\ge5$. When $\gcd(n,6)=1$ then by
\cite[Prop.~6.9]{BGK} for any $1\ne x\in\fA_n$ there exists an $n$-cycle
$y$ such that $\langle x,y\rangle=\fA_n$.   Suppose that $n=2m$ is even.
Then choose $e$ with $1 \le e \le 6$ such that $\gcd(m+e,m-e)=1$
and $\gcd(m \pm e, 6)=1$.  Let $y \in \fA_n$ be the product of disjoint
cycles of length $m + e$ and $m - e$.   Clearly, $y$ is not contained
in any imprimitive subgroup.  Thus, for $n \ge 16$, $y$ is contained
in a unique maximal subgroup by \cite{will}.   If $n=6,8,10,12$ or $14$, a
computer check shows that one can take $y$ of order $5, 7, 5, 35$ or $13$
respectively.
Now suppose that $n$ is odd and divisible by $3$.   Let $y \in \fA_{n-1}$ be
chosen as above.   If $n > 16$, it follows as above that $y$ is not contained
in any imprimitive subgroup.   Clearly, given $1\ne x \in S$, we can choose a 
conjugate $y'$ of $y$ such that $\langle x,y'\rangle$ is transitive and so
primitive. It again follows by \cite{will} that $S = \langle x, y'\rangle$.
If $n < 16$, then we need only consider $n = 9$ and $15$.  If $n=15$, take
$y$ to be of order $13$ and if $n=9$, take $y$ of order $7$. It is
straightforward to check that the result holds.
\par
For $S$ simple of classical Lie type, we typically take $y$ the smallest power
of an element as in Table~\ref{tab:classtorus} of order prime to~6. Let first
$S=\PSL_n(q)$ with $n\ge4$ even. The possible maximal subgroups containing
such an element $y$ are given in Lemma~\ref{lem:maxSLn-1}. Note that case~(2)
does not arise since $n$ is even. If $q=r^2$ is a square, then $o(y)$, the
prime to~6 part of $(q^{n-1}-1)/\gcd(n,q-1)=(r^{n-1}-1)(r^{n-1}+1)/(n,q-1)$,
is divisible by a Zsigmondy prime divisor of $(r^{n-1}-1)$, so case~(1) is out
as well. Hence there are at most two maximal subgroups containing~$y$, and
the only cases for which the lemma does not apply are $\PSL_4(2)\cong\fA_8$
and $\PSL_4(4)$. For $S=\PSL_4(4)$ the $(C,63a,85a)$-structure constants are
non-zero, and none of the subgroups in Lemma~\ref{lem:maxSLn} contains
elements of orders $63$ and~85. 
\par
Now let $S=\PSL_3(q)$. Then by Lemma~\ref{lem:maxSLn} there are at most two
maximal overgroups of $y$ (viz. cases~(1) and (4) for $f=3$), so we are done
unless $S\in\{\PSL_3(2)\cong\PSL_2(7),\PSL_3(4)\}$. For the first group,
see \cite[Table 5]{BGK}, for $\PSL_3(4)$ the $(C,5a,7a)$-structure constants
are non-zero.
\par
Next let $S=\PSL_n(q)$ with $n\ge5$ odd. Let $C_2$ denote the conjugacy class
of an element of order $o_1:=(q^n-1)/d(q-1)$, and $C_3$ the class of an
element of order $o_2:=((q^{n-1}-1)/d)_{2,3}'$, where $d=(n,q-1)$. By the
argument in the proof of \cite[Thm.~2.1]{MSW} only the trivial and the
Steinberg character of $S$ take non-zero values on both classes, and thus the
$(C,C_2,C_3)$-structure constant does not vanish for any non-trivial class
$C$ of $S$. Now consider
the subgroups in Lemma~\ref{lem:maxSLn}. Cases~(2) and~(3) are out since $n$
is odd, and case~(1) is excluded as before. The extension field subgroups
$\GL_{n/f}(q^f)$ occur for prime divisors $f|n$. Now note that their order is
not divisible
by a Zsigmondy prime divisor of $q^{n-1}-1$, which exists unless $(n,q)=(7,2)$.
But in the latter case, there is just one overgroup, so we are done again.
\par
Next let $S=\PSU_n(q)$ with $n\ge4$ even, $(n,q)\ne(4,2)$. If $(n,q)$ is none
of $(6,2),(10,2)$, $(4,3),(4,5)$ then by Lemma~\ref{lem:maxSUeven} there is only
one maximal subgroup containing $y$ and we are done. For $(n,q)=(10,2)$ the
argument in \cite[Thm.~2.2]{MSW} shows that the $(C,C_2,C_3)$-structure
constant does not vanish for any non-trivial class $C$ of $S$, where $C_2$
contains elements of order~19 and $C_3$ elements of order
$(q^n-1)/(q-1)=1023$. None of the groups in Lemma~\ref{lem:maxSUeven} contains
elements of order~13. The same may be applied to $\PSU_4(5)$, with elements
of order~7 respectively~13. For $\PSU_6(2)$, the $(C,11a,12f)$-structure
constant is non-zero for any $C\ne\{1\}$, and the class $12f$ has trivial
intersection with $\PSU_5(2)$ as well as with $M_{22}$. For $\PSU_4(3)$, the
$(C,7a,9a)$-structure constants are non-zero and none of the relevant maximal
subgroups contains elements of order~9. 
\par
Now let $S=\PSU_n(q)$ with $n\ge 3$ odd. Firstly, if $n$ is prime and
$(n,q)$ is different from $(3,3),(3,5),(5,2),(9,2)$, then by
Lemma~\ref{lem:maxSUodd}
there is a unique maximal subgroup containing our element $y$, and we are done.
The same holds in fact for $\PSU_3(3)$ and $\PSU_5(2)$. 
Else, we let $C_2$ denote a class of elements of order $(q^n+1)/d(q+1)$ and
$C_3$ a class of elements of order $(q^{n-1}-1)/d$, where $d=\gcd(n,q+1)$.
By \cite[Thm.~2.2]{MSW}, the $(C,C_2,C_3)$-structure constants are non-zero
in $S$. The extension field subgroups in Lemma~\ref{lem:maxSOodd}(1) and~(3)
are ruled out by Zsigmondy. For $\PSU_3(5)$ the $(C,7a,8a)$-structure
constants are non-zero, and no maximal subgroup contains elements of orders~7
and~8.
\par
For $S=\OO_{2n+1}(q)$, we are done by Lemma~\ref{lem:maxSOodd} unless
$n=p=3$ or $(2n+1,q)=(7,5)$. Suppose that $n=p=3$.  In this case, we take $y$
to be of order $(q+1)_{2,3}'$. If $q>3$, then Lemma~\ref{lem:maxSOodd} still
implies that $y$ is in a unique maximal subgroup. For $\OO_7(3)$ we may take
$y$ of order~13, for $\OO_7(5)$ of order~31.
\par
Now let $S=\PSp_{2n}(q)$. Let $y$ be an element of the largest possible order
prime to $6$ dividing $q^n+1$. Let $\mathcal{M}$ denote the set of maximal
subgroups of $S$ containing $y$.   First assume that $q$ is even. If $n > 3$,
it follows by Lemma \ref{lem:maxSp} that all maximal subgroups containing
$y$ contain the centralizer of $y$. Let $z$ be a generator for the centralizer
of $y$. It was shown in \cite[5.8]{BGK} that for any $g\in S$, we have that
$S = \langle g, z' \rangle$ for some conjugate $z'$ of $z$.   Thus, the same
is true for $y$. If $n=2$, we may assume that $q\ne 2$. If $q$ is a square,
then $|\mathcal{M}|=2$ and the result follows.  If $q$ is not a square, then
$|\mathcal{M}| = 3$.  If $q  \ge 8$, it follows by \cite{liesaxl} that 
$|s^S \cap M| \le (4/3q)|s^S| \le |s^S|/6$, whence the result.   If $n=3$,
then the same argument applies for $q > 4$.  See \cite[Table 2]{BGK} for
the case $q=4$. For $\PSp_6(2)$ a computer check shows that we can take for $y$
an element of order~7.
\par
Now assume that $q$ is odd.    Let $s$ be a nontrivial element of $S$.  Suppose
that $s^S$ is contained in the union of the maximal subgroups containing $y$. 
Excluding the cases with $(n,q)= (4,3), (6,3), (6,5)$, the only
maximal subgroups
containing $y$ are one  for each prime divisor of $n$ and if $n$ is odd, the 
normalizer of $\SU_n(q)$.    Suppose that $s^S$  intersects  $b \ge 3$
of the  maximal subgroups containing $y$.   It follows that $s$ is contained
in a conjugate of the normalizer $\Sp_{2n/f}(q^f)$ for some prime $f > b$.
By \cite[2.1 and  2.3]{BGK}, $|s^S \cap M| \le (4/3q^{f-1})|s^S|$ for any
proper subgroup $M$ of $S$. Thus,
$$|\bigcup_{M \in \mathcal{M}} (s^S \cap M)|\le (4(f-1)/3q^{f-1}) |s^S|<|s^S|,$$
a contradiction. For $\PSp_4(3),\PSp_6(3),\PSp_6(5)$ a computer check shows
that we may take $y$ of order~5, 13, 31 respectively.
\par
For $S=\OO_{2n}^-(q)$, let $C_2$ denote the class of elements of order
$o_1:=(q^n+1)/(4,q^n+1)$ and $C_3$ a class of semisimple elements of order
$o_2:=(q^{n-1}+1)_{2,3'}$. The arguments in the proof of \cite[Thm.~2.5]{MSW}
show that only the trivial and the Steinberg character do not vanish on both
$C_2,C_3$. Thus, the structure constant $n(C,C_2,C_3)$ is non-zero for any
non-trivial class $C$ of $S$. None of the subgroups listed in
Lemma~\ref{lem:maxSO-} and Table~\ref{tab:o-} contain elements of orders
$o_1,o_2$, except possibly when no Zsigmondy prime for $q^{n-1}+1$ exists,
that is, when $(2n,q)=(8,2)$. For the latter group, the $(C,17a,21a)$-structure
constants are non-zero, and no maximal subgroup contains elements of orders
17 and 21.
\par
Next consider  $S=\OO_{2n}^+(q), n \ge 4$. If $n=4$ and $q > 4$, the result
follows by \cite[Lemma 5.15]{BGK} (the element chosen there has order prime
to $6$).  If $n=4=q$, the result follows by \cite[Table 3]{BGK}.
Similarly, the result follows by \cite[Lemmas 5.13,  5.14]{BGK} if  $n > 4$.
For $S=\OO_8^+(3)$ we may take $y$ of order~13. The group $\OO_8^+(2)$ is
an exception.
\par
For $S=\PSL_2(q)$, let $C_2$ be a class of elements of order~$p$ and $C_3$ a
class of elements of order $(q+1)/(2,q+1)$. Then the $(C,C_2,C_3)$-structure
constant is non-zero. For $q\ge5$ prime to~6, no proper subgroup of $G$
contains elements of orders~$p$ and $(q+1)/(2,q+1)$. For $q=p^f$ with $p=2,3$,
let $C_2$ be a class of elements of order $(q-1)/(2,q-1)$ instead; again the
structure constants do not vanish, one of $C_2,C_3$ contains elements of
order prime to~6, and no proper maximal subgroup does contain elements from
both classes unless $q=9$, so $S=\fA_6$, which was already considered.
\end{proof}

We end this section by demonstrating on the example of $E_8(q)$ how our
methods can also be used to show that simple groups admit a so-called unmixed
Beauville-structure (see for example \cite{shelly}):

\begin{thm}   \label{thm:beau}
 The simple groups $E_8(q)$, where $q$ is any prime power, admit an unmixed
 Beauville structure.
\end{thm}

\begin{proof}
By \cite[Def.~1.1]{shelly} it suffices to show that $E_8(q)$ has two
generating systems $(x_1,y_1)$, $(x_2,y_2)$ such that the orders of
$x_1,y_1,x_1y_1$ are pairwise prime to those of $x_2,y_2,x_2y_2$.
\par
By Proposition~\ref{prop:excbig} we already know that there exists a conjugacy
class $C$ of $G$ such that $G$ is generated by $x_1,y_1\in C$ with
$(x_1y_1)^{-1}\in C$, containing elements of order $n:=\Phi_{30}(q)$. Now let
$C_1$ be a class of elements of order $\Phi_{24}(q)$, $C_2$ a class of
elements of order $\Phi_{20}(q)$, and $C_3$ a class containing the product
of a (long) root element $x$ with a semisimple element of order $\Phi_{14}$
(from the subgroup $E_7(q)$ lying in $C_G(x)$). Note that all three classes
contain elements of order prime to~$n$. Using Lemma~\ref{lem:Lvalues} it is
immediate to see that at most unipotent characters might take non-zero values
on the two semisimple classes simultaneously. Moreover, from the tables in
\cite[13.9]{Ca} one sees that in fact only four unipotent characters have
this property. Apart from the trivial character, this is the Steinberg
character (which vanishes on the non-semisimple class $C_3$) and two
further characters of degree $>q^{100}$. Since the elements in all three
classes have centralizer order less than $q^{10}$, we conclude that
$n(C_1,C_2,C_3)\ne0$. \par
The result of Cooperstein \cite{Co2} shows that no proper subgroup of $G$
containing long root elements has order divisible by $\Phi_{20}(q)$ and
$\Phi_{24}(q)$.
\end{proof}

\vskip 1pc
\noindent
\textbf{Acknowledgements:} We wish to thank Thomas Breuer for checking some
of the maximal subgroups of classical groups, Frank L\"ubeck for help in
dealing with $F_4(3)$, Klaus Lux for determining the degrees of irreducible
representations of $\SU_3(8)$ over $\FF_2$ and Pham Huu Tiep for pointing
out two results on the alternating groups and their covers.


\end{document}